\documentclass[graybox]{svmult}

\usepackage{type1cm}        
\usepackage{makeidx}         
\usepackage{graphicx}        
\usepackage{multicol}        
\usepackage[bottom]{footmisc}

\usepackage{newtxtext}       %
\usepackage[varvw]{newtxmath}       

\makeindex             

\usepackage{adjustbox}
 \usepackage{natbib}
\usepackage{tikz}
\usetikzlibrary{cd}

\usepackage{mathrsfs}
\usepackage[all]{xy}
\usepackage{lscape}
\usepackage{bbm}

\usepackage{bussproofs}

\EnableBpAbbreviations

\newcommand{\CTX}{\mathsf{Contexts}}
\newcommand{\FORM}{\mathsf{Formulae}}

\newcommand{\MSEQ}{\mathsf{MultiSequents}}

\newcommand{\SEQ}{\mathsf{Sequents}}
\newcommand{\TERM}{\mathsf{Terms}}

\newcommand{\B}[1]{\mathbf{#1}}

\newcommand{\TT}[1]{\mathtt{#1}}

\newcommand{\C}[1]{\mathcal{#1}}
\newcommand{\BB}[1]{\mathbb{#1}}

\newcommand{\F}[1]{\mathfrak{#1}}

\newcommand{\To}{\Rightarrow}

\definecolor{color0}{HTML}{4682B4}

\usepackage{turnstile}

\begin{document}

\title{A New Conjecture about Identity of Proofs}

\author{Paolo Pistone}
\institute{Paolo Pistone \at Dipartimento di Informatica-Scienza e Ingegneria, Universit\`a di Bologna, Via Faustino Malaguti 1/D, 40126 Bologna, Italy, \email{paolo.pistone2@unibo.it}}

\maketitle

\abstract{A central problem in proof-theory is that of finding criteria for identity of proofs, that is, for when two distinct formal derivations can be taken as denoting the same logical argument. In the literature one finds criteria which are either based on proof normalization (two derivations denote the same proofs when they have the same normal form) or on the association of formal derivations with graph-theoretic structures (two derivations denote they same proof when they are associated with the same graph). 
In this paper we argue for a new criterion for identity of proofs which arises from the interpretation of formal rules and derivations as natural transformations of a suitable kind. We show that the naturality conditions arising from this interpretation capture in a uniform and elegant ways several forms of ``rule-permutations'' which are found in proof-systems for propositional, first- and second-order logic.}

\section{Introduction}

It is often argued that a distinctive feature of logical arguments is that these are invariant under any permutation of the domain of objects they involve. So, if an argument regarding complex numbers, elliptic curves and projective spaces  
is valid only by virtue of logic, the argument remains valid if we replace its algebraic objects by tables, chairs and beer mugs (to quote a famous example of Hilbert's). 

By the way, at least since Aristotle's \emph{Prior Analytics}, logicians are used to describing logical arguments in a \emph{schematic} way, that is, by employing variables or meta-variables standing for the possible objects, terms, or propositions that may occur in them. An ancient example every logician knows is the Barbara syllogism: if all $B$s are $C$s, and all $A$s are $B$s, then all $A$s are $C$s.  
Syllogisms contain letters standing for terms; instead, the formal languages we employ today in logic contain variables for different kinds of entities: for \emph{individuals}, for \emph{predicates}, for \emph{propositions} (or 0-ary predicates), and even for entities of higher orders.

The development of proof-theory, the branch of logic that studies logical arguments, their shapes, their dynamics and their identity conditions, has highlighted this schematic nature even further. Purely logical proofs are often formalized by means of sequent calculus or natural deduction rules, whose formulation not only involves variables for individuals and predicates, but also schematic letters, or meta-variables, for arbitrary formulae, for {contexts} (lists or multisets of formulae), and sometimes even for derivations. For instance, typical sequent calculus rules like those below 
\begin{center}
\adjustbox{scale=0.7}{\begin{minipage}{1.4\textwidth}\begin{align}
\AXC{$\Gamma \vdash A$}
\AXC{$\Delta \vdash B$}
\RL{$\land$R}
\BIC{$\Gamma, \Delta \vdash A\land B$}
\DP
\qquad\qquad
\AXC{$\Gamma \vdash A$}
\AXC{$\Gamma \vdash B$}
\RL{$\land$R'}
\BIC{$\Gamma \vdash A\land B$}
\DP
\end{align}\end{minipage}}\end{center}
use schematic letters for formulae ($A$ and $B$) and for contexts ($\Gamma$, $\Delta$), where the latter play an essential role, since they make the two rules non-equivalent (as they correspond to the multiplicative and additive formulation of the right-rule for conjunction). 

%


\subparagraph{The Problem of Identity of Proofs}

A central problem in proof-theory is that of understanding the conditions under which distinct formal derivations may represent the same logical argument. Following \citep{Dosen2003, Tranchini2021}, this corresponds to the problem of defining the \emph{identity of proofs} relation.

Indeed, it is often 
the case that semantic or just intuitive considerations lead to take two distinct derivations (that is, derivations which might differ either by the kind or by the order of the rules they employ) as formalizations of ``the same'' proof. For instance, in the proof-theoretic semantic tradition \citep{Prawitz1971, Prawitz1973, Dosen2003} it is common to take usual cut-elimination/normalization procedures as ways of making the content of a proof more explicit: this implies that if a formal derivation $\Pi$ can be converted to $\Pi'$ using such procedures, then $\Pi$ and $\Pi'$ can be seen as different stages of the process of making the same logical argument explicit.
To make another example, under the so-called \emph{functional interpretation} \citep{CH, Girard1989}, where a proof of, say, $A\supset B$ is interpreted as some {function} going from objects of type $A$ to objects of type $B$, two distinct derivations denote the same proof when they correspond to the same function.

Nevertheless, finding convenient characterizations of the identity of proofs relation in purely syntactic terms can be difficult, since the property of having the same denotation is a genuinely semantic one, and in several cases it needs not even be algorithmically decidable \citep{Okada1999, PistoneCSL}.

The criteria for identity of proofs in the literature generally belong to two families. Firstly, those \citep{Prawitz1973,Dosen2003} where one considers \emph{reduction} relations over formal derivations (like those arising from normalization in natural deduction), and one stipulates that two derivations $\Pi,\Pi'$ denote the same proof when $\Pi$ and $\Pi'$ have the same \emph{normal form} under such reductions (this is what \citep{Dosen2003} calls the ``normalization conjecture''). 
Secondly, those where one associates a formal derivation $\Pi$ with some underlying graph-theoretic structure $\C G(\Pi)$ that is supposed to capture the ``essence'' of $\Pi$ independently from the ``bureaucracy'' \citep{pregoi, Girard1989} of sequent calculus rules. For instance, $\C G(\Pi)$ might be what Lambek calls the \emph{generality} of $\Pi$ \citep{Lambek1968, Lambek1969}, which essentially corresponds to the \emph{graphs} of \citep{EKelly1966}, or the (multiplicative) \emph{proof nets} of \citep{LLSS}. One then stipulates that that two derivations $\Pi,\Pi'$ denote the same proof when $\C G(\Pi)=\C G(\Pi')$ (this is what \citep{Dosen2003} calls the ``generality conjecture'').

More recently, much research in mathematical logic has focused on \emph{canonical} proof-systems, that is, on formalizations such that logical arguments that only differ by inessential syntactic aspects (like free permutations of rules, see Section \ref{sec:permutations}) are represented by the same formal entity. 
Among such proof-systems we can mention those based on proof nets \citep{LLSS}, string diagrams \citep{Mellies2012}, focusing \citep{Scherer2017}, deep inference \citep{guglielmi2010}, and combinatorial proofs \citep{Hughes2006, Hughes2018, Strass2021}.

Unfortunately, these approaches are not always capable of fully capturing the identity of proof relation, as it results from purely semantic considerations. In some cases, this can be done only at the price of making the syntax too complex to be effectively surveyable (i.e.~of making it impossible to check if a derivation is correct in polynomial time \citep{Heij2014}).
%
%
%
%
%
%
Things are specially problematic in presence of arithmetical primitives \citep{Okada1999} or higher-order quantification (see \citep{PistoneCSL}). 
Furthermore, in some cases approaches based on normal form and on graphs may even lead to incompatible identity of proofs relations (see \citep{Dosen2003}, p.~491).

\subparagraph{Schematic Arguments and Naturality Conditions}

A central thesis of this article is that a finer mathematical formalization of schematic logical arguments is a key ingredient to understand identity conditions for proofs. 
%
%
%
Intuitively, a schematic argument $\Pi( v)$, involving some (possibly meta-)variable $ v$, 
has a finite number of schematic premisses $P_{1}(  v),\dots, P_{k}( v)$ (which might be either formulas or sequents) and one (or more) schematic conclusions $C_{1}( v),\dots, C_{h}(v)$. 
If the schematic letter $v$ ranges over some domain $D$, then this means that for any choice of an element $d$ of $D$, one can obtain a ``concrete'' argument $\Pi(d)$ of premisses $P_{1}(d),\dots, P_{k}(d)$ and conclusions $C_{1}(d),\dots, C_{h}(d)$. 
Importantly, the schematic nature of $\Pi(v)$ not only ensures that one has a family $(\Pi(d))_{d\in D}$ of arguments indexed by the elements of the domain; it also ensures that all such concrete arguments $\Pi(d),\Pi(e),\dots$ are constructed ``in the same way'', that is, using instances of the same rules: they are all instances of the unique schematic $\Pi(v)$. 

Schematic constructions of this kind abound in mathematics, and are often described as \emph{natural transformations}. 
For instance, for any group $G$ one can define a homomorphism (actually, an \emph{iso}morphism) $\varphi_{G}:G\longrightarrow \mathrm{Hom}(\BB Z,G)$ between $G$ and the group of homomorphisms between the additive group $\BB Z$ of integers and $G$, where $\varphi_{G}(g)(z)=g^{z}=g\cdot g\cdot \dots \cdot g$.
The operation $\varphi_{G}$ is schematic in that, for any group $G$, it produces a map which is always defined ``in the same way'', that is, independently of the particular features of $G$. This uniformity of $\varphi_{G}$ can be expressed in mathematical terms by observing that $\varphi_{G}$ is a natural transformation. This means that if $G,H$ are two groups, and $\alpha:G\longrightarrow H$ is some group homomorphism (inducing another group homomorphism $\alpha': \mathrm{Hom}(\BB Z,G)\longrightarrow \mathrm{Hom}(\BB Z,H)$ given by $\alpha'(f)(z)=\alpha(f(z))$), then the same function is obtained by either first applying $\alpha$ and then $\varphi_{H}$, or by first applying $\varphi_{G}$ and then $\alpha'$:
\begin{center}
\adjustbox{scale=0.7}{
\begin{minipage}{1.4\textwidth}
\begin{equation}\label{eq:natu0}
\begin{tikzcd}
G \ar{rr}{\varphi_{G}}\ar{d}[left]{\alpha} & &  \mathrm{Hom}(\BB Z,G) \ar{d}{\alpha'} \\
H \ar{rr}[below]{\varphi_{H}} & & \mathrm{Hom}(\BB Z,H)
\end{tikzcd}
\end{equation}
\end{minipage}}
\end{center}
Intuitively, the diagram above ensures that the schematic operation $\varphi_{G}$ is so uniform that one can freely decide to first move from $G$ to $H$ and then apply it, or to go the other way round. 
At the same time, the diagram above is nothing but a visual illustration of the equation $\alpha'(\varphi_{G}(\_))= \varphi_{H}(\alpha(\_))$, which holds since 
$\alpha'(\varphi_{G}(g))(z)=\alpha(g^{z})=\alpha(g)^{z}=\varphi_{H}(\alpha(g))(z)$. Thus, naturality conditions not only capture a notion of uniformity for schematic operations, but they do this by means of an essentially equational criterion. 
%

The goal of this paper is to show that the schematic nature of logical rules and logical arguments 
can be expressed by means of naturality conditions like \eqref{eq:natu0}, and that such equational conditions provide ways to express principles of identity of proofs in a uniform and elegant way.

%

%
The idea that logical arguments correspond to natural transformations is not new. For instance, graphical entities like proof-nets or string diagrams commonly used to formalize proofs in linear logic do correspond to natural transformations of some kind (due to their close connection with the functorial calculus from \citep{EKelly1966}).  
%
%
%
Moreover, in the study of second-order proof-systems and their connection with \emph{polymorphic} lambda-calculi \citep{Bainbridge1990, Girard1992, StudiaLogica, PistoneCSL}, naturality conditions have been widely investigated
as a means to capture \emph{program equivalence} (that is, the property by which two programs compute the same function). In particular, the so-called notion of \emph{parametricity} \citep{Reynolds1983, Bainbridge1990}, corresponding to the idea that a proof of a universally quantified formula/type $\forall p.A$ coincides with a schematic proof of $A[p]$, 
can in some cases (see \citep{Hasegawa2009, StudiaLogica}) be expressed in an elegant way by means of naturality (or, more generally, of \emph{di}naturality \citep{MacLane}) conditions.

In \citep{StudiaLogica} the author showed that the naturality conditions arising from parametricity correspond, in proof-theoretical terms, to the validity of a class of rule permutations which can be used to study the identity of proof relation in second-order logic (see also \citep{PistoneCSL} for some decidability results on this relation). 
In this paper we push this perspective further, by introducing new naturality conditions, and consequently, new principles of identity of proofs, also for propositional and first-order logic.


%
%
%

%
%

\subparagraph{Outline of the Paper}

In Section 
\ref{sec:categories} we introduce some basic categorical language for proof-theory (like the categories $\CTX$ and $\FORM$ of contexts and formulae) and we recall the basic notions of functor and natural transformation. 

In Section 
\ref{sec:permutations} we show that the criterion for identity of proofs arising from free permutation of rules in the sequent calculus is captured by a naturality condition with respect to the schematic letters for contexts in sequent calculus rules.

In Section
\ref{sec:disjunction} we show that the criterion of identity of proofs arising from usual permutative conversions in natural deduction is captured by a naturality condition with respect to the schematic letters for formulae in the conclusions of so-called generalized elimination rules.

In Section  
\ref{sec:parametricity} we discuss the {parametric} interpretation of second-order proofs \citep{Bainbridge1990, Girard1992},
and we show that it can be used to justify rule permutations involving rules for the second-order quantifiers. We then show that invariance under such permutations is captured by 
a naturality condition with respect to propositional variables in second-order derivations. This section essentially surveys results from \citep{StudiaLogica, SL2}.

In Section 
\ref{sec:unification}
we provide a sketch of a ``parametric'' interpretation of first-order logic, and we introduce a naturality condition with respect to individual variables in first-order derivations. We then deduce from this condition the validity of some new rule permutations for first-order logic.

\section{Categories of Derivations and Natural Transformations}\label{sec:categories}

This article is primarily intended for logicians and philosophers with some background in proof-theory. For this reason we will presuppose basic acquaintance with sequent calculus and natural deduction systems, as well as with their dynamics, i.e.~cut-elimination and normalization.
Instead, we do not presuppose knowledge of category theory, and in this section we will introduce all categorical notions that will appear in the text. We will devote special attention to introducing three categories that will be used to model formal derivations.

\subparagraph{Preliminaries on Categories}

Usual logical languages rely on a twofold ontology: on the one hand one has \emph{propositions} (and, more generally, \emph{formulae}), and on the other hand one has \emph{derivations}, which are certain objects having one formula as its \emph{conclusion} and a list of formulae as \emph{assumptions}. 
In proof-theory, the branch of logic in which derivations are the central object of investigation, it is common to consider also a third kind of entities, beyond formulae and derivations: these are the \emph{relations} existing between derivations with same conclusion and same assumptions. 
A typical example is provided by the conversion rules defining cut-elimination procedures in sequent calculus or normalization in natural deduction. Other examples are provided by {equivalence} relations over derivations, which might either be generated from conversion rules or be induced by some denotational semantics, and which will constitute our main object of investigation in this paper.

 

A category $\C C$ is an algebraic structure relying on a similar threefold ontology: firstly, one has a class of \emph{objects} $A,B,C,\dots$; secondly, one has a class of \emph{arrows} $f,g,h,\dots$ between objects, i.e.~where each arrow has one \emph{target} object and one \emph{source} object (so that it makes sense to write $f:A\longrightarrow B$); we will indicate as $\C C( A,B)$ the hom-sets of $\C C$, that is, the sets of arrows with source $ A$ and target $B$; thirdly, one has a class of \emph{diagrams} between arrows, i.e.~of equations between arrows having the same source and same target.

A category can be seen as a generalization of the usual algebraic notion of \emph{monoid}. We recall that a monoid is a set $M$ equipped with a binary associative operation $*: M\times M\longrightarrow M$ together with a unit $1\in M$ (i.e.~such that $x*1=1*x=x$ holds for all $x\in M$). 
Similarly, a category comes with a binary associative operation $\circ : \C C (B,C)\times \C C(A,B)\longrightarrow \C C(A,C)$ called \emph{composition} (where associativity reads as the diagram $h\circ (g\circ f)=(h\circ g)\circ f$ for all 
arrows $f:A\longrightarrow B, g:B\longrightarrow C, h:C\longrightarrow D$) and, for any object $A$, with a special arrow $1_{A}:A\longrightarrow A$, called the \emph{identity on $A$}, satisfying ``unit'' equations $f\circ 1_{A}=f$ and $1_{A}\circ g=g$, for all $f:A\longrightarrow B$ and $g:C\longrightarrow A$.

A monoid $(M, *, 1)$ can thus be seen as a special case of a category, namely one with precisely one object $\star$: any $x\in M$ corresponds to an arrow $x:\star\longrightarrow \star$, and $*$ and $1$ correspond to composition and to the identity arrow $1_{\star}$, respectively.

%

%
%

One of the simplest category with more than one object one can think of is the category $\mathsf{Set}$ of sets: its objects are just sets, its arrows are functions between sets, and diagrams correspond to plain equations between functions. Similarly, most algebraic or geometric structures give rise to suitable categories: one has a category $\mathsf{Group}$ of groups and their homomorphisms, a category $\mathsf{Ring}$ of rings and their homomorphisms, a category
 $\mathsf{Top}$ of topological spaces and continuous functions, a category $\mathsf{Vec}$ of vector spaces and linear functions, etc.  

There are two basic operations that can be used to construct new categories from known ones, and that we will exploit throughout the text. Firstly, for any category $\C C$ we will consider its \emph{opposite} category $\C C^{\mathsf{op}}$. This category has the same objects and arrows as $\C C$, but sources and targets of the arrows are inverted: an arrow $f:A\longrightarrow B$ in $\C C^{\mathsf{op}}$ is the same as an arrow $f:B\longrightarrow A$ in $\C C$ (we leave it to the reader to check that all the diagrams of a category still hold under this reversal). So for example $\mathsf{Set}^{\mathsf{op}}$ is the category with objects being sets and arrows between two sets $A$ and $B$ being functions $f: B\longrightarrow A$ (i.e.~$\mathsf{Set}^{\mathsf{op}}(A,B)=\mathsf{Set}(B,A)$). 
%
%

Secondly, we will make large use of \emph{quotients}: suppose $\C C$ is a category and $\equiv_{A,B}\subseteq \C C(A,B)\times \C C(A,B)$ is a family of equivalence relations on the hom-sets of $\C C$; if the equivalences $\equiv_{A,B}$ satisfy the congruence conditions below
\begin{center}
\adjustbox{scale=0.7}{
\begin{minipage}{1.4\textwidth}
\begin{align}\label{condi}
f\equiv_{A,B}f' ,g\equiv_{B,C}g' \ \longrightarrow \ g\circ f \equiv_{A,C}g'\circ f' 
\end{align}\end{minipage}}\end{center}
then we can define a new category ${\C C}/{\equiv}$, that we call the \emph{quotient of $\C C$ under $\equiv$}, having the same objects as $\C C$ and where an arrow from $A$ to $B$ is a class of equivalence $[f]_{\equiv_{A,B}}$ of arrows from $A$ to $B$ in $\C C$.

\begin{remark}\label{rem:multicat}
Just as a derivation has one conclusion but may have a finite number of assumptions, one can adjust the notion of category to this idea: a \emph{multicategory} is defined similarly to a category, but instead of arrows one has \emph{multi-arrows} which have, instead of one single source object, a list of source objects. We will indicate a multi-arrow in a multicategory as
$f: A_{1},\dots, A_{n}\longrightarrow B$ or, more concisely, as $f:\vec A\longrightarrow B$.

In a multicategory one has identity multi-arrows $1_{A}:A\longrightarrow A$ (where $A$ can be seen as the 1-object list $\vec A=A$), and composition of multi-arrows is defined as follows: given multi-arrows  $g: B_{1},\dots, B_{n}\longrightarrow C$ and  $f_{1}:\vec A_{1}\longrightarrow B_{1},\dots, f_{n}: \vec A_{n}\longrightarrow B_{n}$, the {composition} multi-arrow is $g\circ ( f_{1},\dots, f_{n}): \vec A_{1},\dots, \vec A_{n}\longrightarrow C$.
Multi-arrows are required to satisfy identity and associativity diagrams analogous to those of a category (for more details, the reader can look at \citep{Leinster2004}, p.~35).
\end{remark}
%
%
%

\subparagraph{Three Categories of Derivations}

Categories arise in a natural way in practically any field of mathematics, and logic makes no exception. We are now going two present a few useful categories in which arrows correspond to formal derivations. 

Let us fix a logical system $\C S$, be it in sequent calculus or natural deduction\footnote{In case of a sequent calculus system, we make the further assumption that $\C S$ includes the usual structural rules $(e),(w)$ and $(c)$ of \emph{exchange}, \emph{weakening} and \emph{contraction}.}, so that we can properly speak of formulae, derivations between formulae and relations between derivations of the same formulae.

We make the further assumption that $\C S$ comes with a well-defined cut-elimination (or normalization) procedure, so that any derivation can be transformed into one in normal form in a finite number of steps. 
The cut-elimination/normalization procedure defines a relation between derivations, and we let $\equiv^{\mathrm{cut}}$ indicate the smallest equivalence relation containing this relation (that is, its reflexive, transitive and symmetric closure). 

In the following lines we mostly rely on sequent calculus, but the whole discussion can easily be converted to natural deduction.

The first category we consider is the \emph{category of formulae of $\C S$}, noted $\FORM_{\C S}$.
Actually, we consider a whole family of categories $\FORM_{\C S}(\vec C)$, where $\vec C=C_{1},\dots, C_{k}$ is some fixed list of formulae. Intuitively, the category $\FORM_{\C S}(\vec C)$ captures reasoning ``under the assumptions'' $\vec C$.\footnote{When the list $\vec C$ is empty we simply write $\FORM_{\C S}$ for $\FORM_{\C S}(\vec C)$.}
 More formally, the objects of $\FORM_{\C S}(\vec C)$ are, as one might expect, the formulae of $\C S$; an arrow between formulae $A$ and $B$ is a derivation $\Pi$ of the sequent $\vec C,A\vdash B$ (we note this as $\Pi: [\vec C], A\vdash B$). 
Diagrams are interpreted as equations between derivations which hold when equality is replaced by the equivalence relation $\equiv^{\mathrm{cut}}$.

For any formula $A$, the identity arrow $1_{A}$ in $\FORM_{\C S}(\vec C)$ is given by the obvious derivation of $[\vec C],A\vdash A$. The composition of derivations $\Pi:[ \vec C],A \vdash B$ and $\Sigma: [\vec C],B \vdash C$ is the derivation of $[\vec C],A \vdash C$ obtained by applying an instance of the $cut$-rule to $\Pi$ and $\Sigma$ (followed by a few contraction rules) as illustrated below:\begin{center}
\adjustbox{scale=0.7}{\begin{minipage}{1.4\textwidth}\begin{align}
\AXC{$\Pi$}
\noLine
\UIC{$[\vec C], A \vdash B$}
\AXC{$\Sigma$}
\noLine
\UIC{$[\vec C], B \vdash C$}
\RL{$cut$}
\BIC{$[\vec C], [\vec C], A \vdash C$}
\doubleLine
\RL{$c$}
\UIC{$[\vec C], A \vdash C$}
\DP\end{align}\end{minipage}}\end{center}

The reader can check that identity and associativity diagrams do hold under $\equiv^{\mathrm{cut}}$.

%
%
%
The second category that we consider is the \emph{category of contexts of $\C S$}, noted $\CTX_{\C S}$. The 
objects of $\CTX_{\C S}$ are, indeed, contexts, that is, finite lists of formulae $\vec A=A_{1},\dots, A_{n}$ of $\C S$; an arrow between two contexts $\vec A$ and $\vec B=B_{1},\dots, B_{n'}$ (noted $\vec \Pi:\vec A \vdash \vec B$) is a list $\vec \Pi=\Pi_{1},\dots, \Pi_{n'}$ of $n'$ derivations of the sequents $\vec A \vdash B_{1},\dots, \vec A \vdash B_{n'}$; as for $\FORM_{\C S}$, diagrams are interpreted as equations holding up to $\equiv^{\mathrm{cut}}$.

For any context $\vec A$, the identity morphism $1_{\vec A}: \vec A \vdash \vec A$ of $\CTX_{\C S}$ is made by the list of obvious derivations of $\vec A\vdash A_{i}$; the composition of (lists of) derivations $\vec \Pi: \vec A\vdash \vec B$ and $\vec \Sigma: \vec B \vdash \vec C$ is again obtained by applying (a few instances of) the $cut$-rule: it is the list of derivations $\vec \Theta=\Theta_{1},\dots, \Theta_{m}$ where $\Theta_{j}: \vec A \vdash C_{j}$ is defined as follows:

\begin{center}
\adjustbox{scale=0.7}{\begin{minipage}{1.4\textwidth}
\begin{align}
\Theta_{j} \quad = \quad
\AXC{$\Sigma_{j}$}
\noLine
\UIC{$B_{1},\dots, B_{n}\vdash C_{j}$}
\AXC{$\Pi_{n}$}
\noLine
\UIC{$\vec A \vdash B_{n}$}
\RL{$cut$}
\BIC{$\vec A, B_{1},\dots,B_{n-1}\vdash C_{j}$}
\noLine
\UIC{$\ddots$}
\noLine
\UIC{$\vec A, \dots, \vec A, B_{1}\vdash C_{j}$}
\AXC{$\Pi_{1}$}
\noLine
\UIC{$\vec A \vdash B_{1}$}
\RL{$cut$}
\BIC{$\vec A, \dots, \vec A \vdash C_{j}$}
\doubleLine
\RL{$c$}
\UIC{$\vec A \vdash C_{j}$}
\DP
\end{align}\end{minipage}}\end{center}

%

Finally, we consider a third category, the \emph{category of sequents of $\C S$}, noted $\SEQ_{\C S}$.
As for the category of formulae, we actually consider a family of categories $\SEQ(\vec C)$, indexed by a finite list of formulae $\vec C=C_{1},\dots, C_{k}$ to be seen as assumptions.\footnote{As in the case of $\FORM_{\C S}$, when $\C C$ is the empty list we simply write $\SEQ_{\C S}$ for $\SEQ_{\C S}(\vec C)$.} 
The objects of $\SEQ_{\C S}(\vec C)$ are sequents $\vec A\vdash A $ (i.e.~pairs made of a context and a formula\footnote{For simplicity we here restrict ourselves to sequents with exactly one formula on the right, but one can easily adapt the definition to sequents with a  context, instead of a formula, on the right.}); an arrow between two sequents $\vec A\vdash A$ and $\vec B\vdash B$ (noted $\Pi: ([\vec C],\vec A\vdash A)\longrightarrow ([\vec C],\vec B\vdash B)$) is a derivation $\Pi$ of $\vec C,\vec B\vdash B$ which might possibly employ the sequent $\vec C,\vec A \vdash A$ as an axiom. We indicate such extra-axioms with a dashed line as follows:
\begin{center}
\adjustbox{scale=0.7}{\begin{minipage}{1.4\textwidth}
\begin{align}
\AXC{}
\dashedLine
\UIC{$[\vec C], \vec A \vdash A$}
\DP
\end{align}\end{minipage}}\end{center}
Diagrams are always interpreted as equations holding under $\equiv^{\mathrm{cut}}$.

The identity $1_{\vec A\vdash A}$ is the derivation only consisting of $\AXC{}\dashedLine\UIC{$\vec A \vdash A$}\DP$, and the composition of $\Pi: ([\vec C],\vec A \vdash A)\longrightarrow ([\vec C],\vec B\vdash  B)$ and $\Sigma: ([\vec C],\vec B\vdash  B)\longrightarrow ([\vec C],\vec C\vdash C)$ is the derivation obtained by replacing each occurrence of $\AXC{}\dashedLine\UIC{$[\vec C],\vec B\vdash  B$}\DP$ in $\Sigma$ by $\Pi$.

\begin{remark}
We can also present $\SEQ_{\C S}(\vec C)$ as a multicategory, that we will indicate as $\MSEQ_{\C S}(\vec C)$: in this case a multi-arrow 
$\Pi: ([\vec C], \vec A_{1}\vdash A_{1};\dots; [\vec C], \vec A_{n}\vdash A_{n}) \longrightarrow ([\vec C], \vec B \vdash B)$ is a derivation possibly using any of the axioms below
\begin{center}
\adjustbox{scale=0.7}{\begin{minipage}{1.4\textwidth}
\begin{align}
\AXC{}
\dashedLine
\UIC{$[\vec C], \vec A_{1} \vdash A_{1}$}
\DP
\qquad
\dots
\qquad
\AXC{}
\dashedLine
\UIC{$[\vec C], \vec A_{n} \vdash A_{n}$}
\DP
\end{align}\end{minipage}}\end{center}
\end{remark}

The categories $\SEQ_{\C S}(\vec C)$ also provide an alternative way of representing natural deduction derivations: we can take a sequent $\vec A \vdash A$ as representing the set of natural deduction derivations of $A$ with open assumptions $\vec A$ (and $\vec C$);  an arrow $\Pi: ([\vec C], \vec A \vdash A) \longrightarrow ([\vec C], \vec B \vdash B)$ can be taken as representing a natural deduction derivation ``with a hole''\footnote{This is what in the literature on $\lambda$-calculus is called a \emph{term context}, that is, a term $\TT C[\ ]$ containing one occurrence of the ``hole'' $[\ ]$, such that for any actual $\lambda$-term $t$, $\TT C[t]$ is also a $\lambda$-term (where $\TT C[\ ]$ might possibly bind variables which occur free in $t$).}: 
this means a derivation $\Pi[\Sigma^{?}]$ of conclusion $B$ and open assumptions $[\vec C], \vec B$ constructed starting from some unknown sub-derivation $\Sigma^{?}$ of conclusion $A$ and open assumptions $[\vec C], \vec A$, so that for any actual derivation $\Sigma: [\vec C], \vec A \vdash A$, by putting it in place of the unknown $\Sigma^{?}$ we obtain an actual derivation $\Pi[\Sigma]: [\vec C], \vec B \vdash B$:
%
%
%
%
\begin{center}
\adjustbox{scale=0.7}{\begin{minipage}{1.4\textwidth}
\begin{align}
\Pi[\Sigma^{?}] \quad = \quad
\AXC{$\stackrel{\vec r}{\vec C}$}
\AXC{$\stackrel{\vec s}{\vec A}$}
\noLine
\BIC{$\Sigma^{?}$}
\noLine
\UIC{$ A$}
\AXC{$\stackrel{\vec r}{\vec C}$}
\AXC{$\stackrel{\vec t}{\vec B}$}
\noLine
\TIC{$\Pi$}
\noLine
\UIC{$B$}
\DP
\end{align}\end{minipage}}\end{center}
Observe that the assumptions $\vec A$ need not be open in $\Pi[\Sigma^{?}]$, since they may be discharged by $\Pi[\Sigma^{?}]$.

In this way the identity arrow $1_{\vec A\vdash A}$ of $\SEQ_{\C S}(\vec C)$ corresponds to the derivation $\Pi[\Sigma^{?}]$ only constituted by the ``hole'' $\Sigma^{?}$, and the composition of two derivations with a hole $\Pi[\Sigma^{?}]$ and $\Pi'[\Sigma^{?}]$ is the result of replacing the hole of the second derivation with the first derivation, i.e.~$\Pi'[\Pi[\Sigma^{?}]]$.

\subparagraph{Functors and Natural Transformations}

Algebraic structures usually come with their own notion of \emph{homomorphism} (e.g.~group homomorphisms, ring homomorphisms etc.), and categories make no exception: homomorphisms between categories are called \emph{functors}; concretely, a functor $F: \C C\longrightarrow \C D$ between categories $\C C$ and $\C D$ is given by (1) a map associating any object $A$ of $\C C $ with an object $F(A)$ of $\C D$, and (2) maps $\C C(A,B)\longrightarrow \C D(F(A),F(B))$ associating any arrow $f:A\longrightarrow B$ in $\C C $ with an arrow $F(f):F(A)\longrightarrow F(B)$ in $\C D$, subject to expectable ``homomorphism'' conditions, i.e.~that $F$ preserves identities and composition, expressed by diagrams $F(1_{A})=1_{F(A)}$ and $F(g\circ f)=F(g)\circ F(f)$.

For instance, for any context $\vec C$, a functor $\C I^{\vec C}:\FORM_{\C S}(\vec C)\longrightarrow \SEQ_{\C S}(\vec C)$ can be defined by letting, for any formula $A$, $\C I^{\vec C}(A)$ be the sequent $\vdash  A$ and, for any derivation $\Pi: [\vec C], A \vdash B$ (corresponding to an arrow in $\FORM_{\C S}[\vec C]$), $\C I^{\vec C}(\Pi:[\vec C], A\vdash B)$ be the derivation obtained by cutting the axiom $\AXC{}\dashedLine\UIC{$[\vec C]\vdash  A$}\DP$ with the derivations $\Pi$:
\begin{center}
\adjustbox{center, scale=0.7}{
\begin{minipage}{1.4\textwidth}
\begin{align}
\C I^{\vec C}(\Pi) \quad = \quad 
\AXC{}
\dashedLine
\UIC{$[\vec C]\vdash A$}
\AXC{$\Pi$}
\noLine
\UIC{$[\vec C], A \vdash B$}
\RL{$cut$}
\BIC{$[\vec C] \vdash B$}
\DP
\end{align}\end{minipage}}
\end{center}
corresponding to an arrow in $\SEQ_{\C S}(\vec C)$.

To give a more semantic example, a functor $\C M: \FORM_{\C S} \longrightarrow \mathsf{Set}$ can be thought of as a \emph{denotational model} of $\C S$, following the functional interpretation of proofs: it associates each formula $A$ of $\C S$ with a set $\C M(A)$, and each proof $\Pi:A\vdash B$ with some function $\C M(\Pi):\C M(A)\longrightarrow \C M(B)$ between the associated sets.

A functor $F: \C C^{\mathsf{op}}\longrightarrow \C D$ is the same thing as a \emph{contravariant} functor from $\C C$ to $\C D$, i.e.~one whose direction of functoriality is reversed: an arrow $f\in \C C(A,B)$ is turned into an arrow $F(f)\in \C D(F(B),F(A))$. 
%

%

\begin{remark}\label{rem:multifunctor}
If $\C C$ is a category and $\C D$ is a multicategory (see Remark \ref{rem:multicat}), a functor 
 $F:\C C\longrightarrow \C D$ from $\C C$ to $\C D$ can be defined as if $\C D$ were a category: the image under $F$ of an arrow $f\in \C C(A,B)$ will be an arrow in $\C D$ with source the 1-object list $F(A)$ and target $F(B)$.
\end{remark}

In the next sections we'll see that several ``schematic'' logical expressions (e.g.~a formula depending on a first- or second-order variable, a rule depending on meta-variables for propositions or contexts) can be associated with suitable functors over the categories of formulae, contexts or sequents.

The last concept we must recall, actually the most important for us, is that of a natural transformation, i.e.~of an arrow between functors: suppose $F,G:\C C\longrightarrow \C D$ are functors between categories $\C C,\C D$; a natural transformation between $F,G$ (noted $\theta: F \To G $) is given by a family of morphisms in $\C D$ of the form $\theta_{A}: F(A)\longrightarrow G(A)$, indexed by the objects of $\C C$, and satisfying, for all $f\in \C C(A, B)$, the ``permutation'' condition $\theta_{B}\circ  F(f)= G(f)\circ \theta_{A}$, which in diagram form reads as the square below:
\begin{center}
\adjustbox{center, scale=0.7}{
\begin{minipage}{1.4\textwidth}
\begin{align}\label{eq:naturality}
\begin{tikzcd}[ampersand replacement=\&]
F(A)\ar{d}[left]{F(f)} \ar{rr}{\theta_{A}} \& \& G(A) \ar{d}{G(f)} \\
F(B) \ar{rr}[below]{\theta_{B}} \& \& G(B)
\end{tikzcd}
\end{align}\end{minipage}}
\end{center}
The basic example of natural transformation is the identity $1: F\To F$, with components the identity arrows $1_{A}:F(A)\longrightarrow F(A)$; moreover natural transformations $\theta: F\To G$ and $\eta: G\To H$ can be composed to obtain a natural transformation
$\eta\circ \theta: F\To H$ (where $(\eta\circ\theta)_{A}=\eta_{A}\circ \theta_{A}$). Indeed, the functors between two categories $\C C$ and $\C D$ and their natural transformations form themselves a category. 


\begin{remark}\label{rem:multinat}
If $\C C$ is a category, $\C D$ is a multicategory, and $F_{1},\dots, F_{n},G: \C C\longrightarrow \C D$ are functors (see Remarks \ref{rem:multicat} and \ref{rem:multifunctor}), it makes sense to speak of a natural transformation $\theta: F_{1},\dots, F_{n} \To G$: this is a family of multi-arrows $\theta_{A}: F_{1}(A),\dots, F_{n}(A)\longrightarrow G(A)$ satisfying a naturality condition $\theta_{B}\circ (F_{1}(f),\dots, F_{n}(f))= G(f)\circ \theta_{A}$, 
which in diagram form reads as the square below:
\begin{center}
\adjustbox{center, scale=0.7}{
\begin{minipage}{1.4\textwidth}
\begin{align}\label{eq:multinat}
\begin{tikzcd}[ampersand replacement=\&]
F_{1}(A),\dots, F_{n}(A)\ar{d}[left]{F_{1}(f),\dots, F_{n}(f)} \ar{rr}{\theta_{A}} \& \& G(A) \ar{d}{G(f)} \\
F_{1}(B),\dots, F_{n}(B) \ar{rr}[below]{\theta_{B}} \& \& G(B)
\end{tikzcd}
\end{align}\end{minipage}}
\end{center}

\end{remark}

Our goal in the next sections will be to convince the reader that several notions of identity of proofs can be expressed in an elegant way using permutation conditions like \eqref{eq:naturality} (or \eqref{eq:multinat}). The key insight will be that schematic rules or derivations  (i.e.~containing variables or meta-variables) can be interpreted as natural transformations in a suitable sense.

\section{Equivalence up to Free Permutations of Rules}\label{sec:permutations}

In this section we discuss a first naturality condition for sequent calculus derivations, which captures a criterion of identity of proofs based on a class of rule permutations. 

\subparagraph{Free Permutations of Rules}
From the viewpoint of identity of proofs, the choice between sequent calculus and natural deduction as our preferred formal system is not devoid of consequences. 
Indeed, as is well known, distinct sequent calculus derivations might correspond to the same natural deduction derivation.
For example, in sequent calculus one can construct two distinct (cut-free) derivations of $A,C\supset B, C\vdash A\land B$, as shown below
\begin{center}
\adjustbox{center, scale=0.7}{
\begin{minipage}{1.4\textwidth}
\begin{align}\label{der1}
\AXC{}
\RL{$ax$}
\UIC{$A\vdash A$}
\AXC{}
\RL{$ax$}
\UIC{$B\vdash B$}
\RL{$\land$R}
\BIC{$A,B \vdash A\land B$}
\AXC{}
\RL{$ax$}
\UIC{$C\vdash C$}
\RL{$\supset$L}
\BIC{$A,C\supset B, C\vdash A\land B$}
\DP
\qquad
\AXC{}
\RL{$ax$}
\UIC{$A\vdash A$}
\AXC{}
\RL{$ax$}
\UIC{$B\vdash B$}
\AXC{}
\RL{$ax$}
\UIC{$C\vdash C$}
\RL{$\supset$L}
\BIC{$C\supset B, C\vdash B$}
\RL{$\land$R}
\BIC{$A,C\supset B,C \vdash A\land B$}
\DP
\end{align}
\end{minipage}}
\end{center}

\noindent
By contrast, in natural deduction there is precisely one normal derivation of $A\land B$ on assumptions $A,C\supset B$ and $C$, shown below
\begin{center}
\adjustbox{center, scale=0.7}{
\begin{minipage}{1.4\textwidth}
\begin{align}\label{der11}
\AXC{$\stackrel{p}{A}$}
\AXC{$\stackrel{q}{C\supset B}$}
\AXC{$\stackrel{r}{C}$}
\RL{$\supset$E}
\BIC{$B$}
\RL{$\land$I}
\BIC{$A\land B$}
\DP
\end{align}\end{minipage}}\end{center}
Observe that the derivations in \eqref{der1} are made of the same rules ($ax$), ($\land$R), ($\supset$L), but these are applied in different order.
Let us say that two consecutive occurrences of rules in a sequent calculus derivation are \emph{freely permutable} when the derivation obtained by switching the order of the rules is correct and has the same conclusion.\footnote{In some cases, see \eqref{der4} in Section \eqref{sec:disjunction}, permuting a rule $R_{1}$ over a rule $R_{2}$ might lead to \emph{duplicate} the $R_{2}$ and the sub-derivations on top of it.} The rules ($\land$R) and ($\supset$L) in \eqref{der1} are freely permutable. 
We will see several examples of freely permutable rules in the following pages, the curious reader can find more examples concerning intuitionistic and linear logic in \citep{Strass2019}, p.~11 and in the appendix of  \citep{Hughes2018}.

Let us say that a formal system $\F F$ for some logic $\C S$ is \emph{canonical} when there exists a mapping $\Pi\mapsto \Pi^{\dag}$ from sequent calculus derivations into derivations in $\F F$ such that, whenever two derivations $\Pi$ and $\Sigma$ can be turned one into the other using a finite number of free permutations of rules, the derivations $\Pi^{\dag}$ and $\Sigma^{\dag}$ in $\F F$ coincide.

For the disjunction-free fragment of intuitionistic propositional logic, it is well-known that natural deduction  does provide a canonical formalism (as the example above suggested).

More generally, the search for {canonical} proof-systems has been a recurring theme of research in proof-theory, at least since the development of linear logic \emph{proof nets} \citep{LLSS}. 
However, finding a canonical formalism for a logic can be difficult. Things become delicate already when one adds disjunction (see Section \eqref{sec:disjunction}) or quantifiers (see Sections \eqref{sec:parametricity} and \eqref{sec:unification}) to intuitionistic propositional logic, or if one replaces intuitionistic logic by classical logic.

%

%
%
   Surveying the vast literature on canonical proof-systems certainly goes well beyond the goals of this paper, but we can mention two significant research directions: first, the replacement of natural deduction trees with either suitable notions of graphs (like \emph{proof nets} \citep{LLSS} for classical and intuitionistic \emph{linear} logic, or \emph{combinatorial proofs} \citep{Hughes2006, Strass2021} for full classical logic); secondly, the replacement of standard sequent calculi with \emph{focused} ones (e.g.~see \citep{Scherer2017} for intuitionistic logic with disjunction), i.e.~calculi with a more rigid discipline regulating the admissible ways of concatenating rules.


\subparagraph{Sequents as Functors, Rules as Natural Transformations}

A typical sequent calculus rule, as the one below
\begin{center}
\adjustbox{center, scale=0.7}{
\begin{minipage}{1.4\textwidth}
\begin{align}\label{eq:land}
\AXC{$\Gamma \vdash A$}
\AXC{$\Delta \vdash B$}
\RL{$\land$R}
\BIC{$\Gamma,\Delta\vdash A\land B$}
\DP
\end{align}\end{minipage}}\end{center}
employs (meta-)variables $\Gamma,\Delta,\dots$ for possible contexts (i.e.~lists of formulae). In fact, a concrete instance of ($\land$R) within a derivation would be something like
\begin{center}
\adjustbox{center, scale=0.7}{
\begin{minipage}{1.4\textwidth}
\begin{align}
\AXC{$\vdots$}
\noLine
\UIC{$\vec E \vdash A$}
\AXC{$\vdots$}
\noLine
\UIC{$\vec F \vdash B$}
\RL{$\land$R}
\BIC{$\vec E,\vec F\vdash A\land B$}
\noLine
\UIC{$\vdots$}
\DP
\end{align}\end{minipage}}\end{center}
where $\vec E=E_{1},\dots, E_{k}$ and $\vec F=F_{1},\dots, F_{h}$ are finite lists of actual formulae. Let us call the formulae $A$ and $B$ the \emph{active part} of (this instance of) ($\land$R), and the formulas $E_{1},\dots, E_{k}$ and $F_{1},\dots, F_{h}$ its \emph{inactive part}. Indeed, two consecutive instances of rules $R_{1},R_{2}$ are freely permutable precisely when the formulae in the active part of $R_{1}$ are inactive or absent in $R_{2}$ and the formulae in the active part of $R_{2}$ are inactive or absent in $R_{2}$.\footnote{Actually, a more refined criterion is needed when considering rules containing side conditions on the context (as usual rules for quantifiers or modal operators). } 

We will now show that the equivalence over derivations generated by free permutations of rules can be captured by a naturality condition. More precisely, by observing that the schematic sequents $\Gamma\vdash A$, $\Delta\vdash B$ and $\Gamma,\Delta \vdash A\land B$ occurring in a rule like \eqref{eq:land} yield suitable functors, and that the schematic rule ($\land$R) yields a natural transformations between such functors.

%
%
%

First, let us define derivations ``modulo free permutations'' formally. We will do this by defining a suitable quotient on the category of sequents:
 for any two sequents $\vec A \vdash A$ and $\vec B \vdash B$, let 
$\equiv^{\mathrm{perm}}_{\vec A,A,\vec B,B}$ be the smallest equivalence relations on 
the derivations in $\SEQ_{\C S}$ from $\vec A \vdash A$ to $\vec B \vdash B$ closed under $\equiv^{\mathrm{cut}}$ and free permutations of rules. 
Since composition in $\SEQ_{\C S}$ is given by $cut$ one can check that the equivalences $\equiv^{\mathrm{perm}}_{\vec A,A,\vec B,B}$ satisfy the congruence condition \eqref{condi}. 
We can define then the quotient category ${\SEQ_{\C S}}/{\equiv^{\mathrm{perm}}}$, and in a similar way the quotient multicategory ${\MSEQ_{\C S}}/{\equiv^{\mathrm{perm}}}$.

Now, let us check that sequent calculus rules yield functors whose image is in ${\SEQ_{\C S} }/ {\equiv^{\mathrm{perm}}}$ (and thus in ${\MSEQ_{\C S}}/{\equiv^{\mathrm{perm}}}$ too, see Remark \ref{rem:multifunctor}).

%
%

The schematic expression $\Gamma \vdash A$ gives rise to a contravariant functor
$\Phi_{A}(\Gamma):\CTX^{\mathsf{op}}_{\C S}\longrightarrow {\SEQ_{\C S}}/ {\equiv^{\mathrm{perm}}}$: with any list of formulae $\vec E$ it associates the sequent $\Phi_{A}(\vec E)=\vec E\vdash A$, and with any list of derivations $\vec\Pi:\vec E\vdash \vec F$ it associates a derivation $\Phi_{A}(\vec\Pi):(\vec F \vdash A) \longrightarrow (\vec E \vdash A)$ as shown below:
\begin{center}
\adjustbox{scale=0.7}{\begin{minipage}{1.4\textwidth}\begin{align}
\Phi_{A}(\vec \Pi)\quad = \quad 
\AXC{}
\dashedLine
\UIC{$\vec F=F_{1},\dots, F_{n} \vdash A$}
\AXC{$\Pi_{n}$}
\noLine
\UIC{$\vec E \vdash F_{n} $}
\RL{$cut$}
\BIC{$\vec E, F_{1},\dots, F_{n-1}\vdash A$} 
\noLine
\UIC{$\ddots$}
\noLine
\UIC{$\vec E,\dots, \vec E, F_{1}\vdash A$}
\AXC{$\Pi_{1}$}
\noLine
\UIC{$\vec E\vdash F_{1} $}
\RL{$cut$}
\BIC{$\vec E,\dots, \vec E\vdash A$}
\doubleLine
\RL{$c$}
\UIC{$\vec E \vdash A$}
\DP\end{align}\end{minipage}}\end{center}The schematic expression $\Delta \vdash B$ gives rise to a functor
$\Phi_{B}(\Delta):\CTX^{\mathsf{op}}_{\C S} \longrightarrow {\SEQ_{\C S} }/ {\equiv^{\mathrm{perm}}}$ defined in a similar way. Also the schematic expression $\Gamma,\Delta \vdash A\land B$ gives rise, for any choice of a context $\vec D$, to a functor $\Phi_{A\land B}(\vec D, \Delta): \CTX^{\mathsf{op}}_{\C S} \longrightarrow {\SEQ_{\C S} }/ {\equiv^{\mathrm{perm}}}$ (for simplicity, we only focus on the second variable $\Delta$ and provide a fixed value $\vec D$ for $\Gamma$) that, to any list $\vec E$, associates the sequent $\Phi_{A\land B}(\vec D, \vec E)=\vec D, \vec E\vdash A\land B$ and, to any list of derivations $\vec\Pi:\vec E \vdash \vec F$, associates a derivation $\Phi_{A\land B}(\vec D , \vec\Pi): (\vec D, \vec F\vdash A\land B)\longrightarrow (\vec D,\vec E\vdash A\land B)$ constructed similarly to the case of $\Phi_{A}$.

For example, if we let $\vec E=C\supset B,C$, $\vec F=B$ and 
$\Pi: \vec E \vdash \vec F$ be the derivation below
 \begin{center}
\adjustbox{scale=0.7}{\begin{minipage}{1.4\textwidth}\begin{align}\label{der2}
\Pi\quad = \quad
\AXC{}
\RL{$ax$}
\UIC{$B \vdash B$}
\AXC{}
\RL{$ax$}
\UIC{$C\vdash C$}
\RL{$\supset$L}
\BIC{$C\supset B, C \vdash B$}
\DP\end{align}\end{minipage}}\end{center}then the functorial derivation $\Phi_{A}(\Pi): (\vec F \vdash B)\longrightarrow (\vec E \vdash B)$ corresponds to the one below:
 \begin{center}
\adjustbox{scale=0.7}{\begin{minipage}{1.4\textwidth}\begin{align}
\Phi_{A}(\Pi)\quad = \quad
\AXC{}
\dashedLine
\UIC{$B \vdash B$}
\AXC{$\Pi$}
\noLine
\UIC{$C\supset B,C\vdash B$}
\RL{$cut$}
\BIC{$C\supset B,C\vdash B$}
\DP
\quad \equiv^{\mathrm{cut}}\quad
\AXC{}
\dashedLine
\UIC{$B \vdash B$}
\AXC{}
\RL{$ax$}
\UIC{$C\vdash C$}
\RL{$\supset$L}
\BIC{$C\supset B, C \vdash B$}
\DP\end{align}\end{minipage}}\end{center}
If we let $\vec D=A$, the functorial derivation $\Phi_{A\land B}(\vec D,\Pi): (\vec D, \vec F \vdash A\land B)\longrightarrow (\vec D, \vec E \vdash A\land B)$ is 
 \begin{center}
\adjustbox{scale=0.7}{\begin{minipage}{1.4\textwidth}\begin{align}
\Phi_{A\land B}(\B D,\Pi)\quad = \quad
\AXC{}
\dashedLine
\UIC{$A,B \vdash A\land B$}
\AXC{$\Pi$}
\noLine
\UIC{$C\supset B,C\vdash B$}
\RL{$cut$}
\BIC{$A,C\supset B,C\vdash A\land B$}
\DP
\quad \equiv^{\mathrm{cut}}\quad
\AXC{}
\dashedLine
\UIC{$A,B \vdash A\land B$}
\AXC{}
\RL{$ax$}
\UIC{$C\vdash C$}
\RL{$\supset$L}
\BIC{$A,C\supset B, C \vdash A\land  B$}
\DP\end{align}\end{minipage}}\end{center}

%
%
%
%
%
%

\begin{remark}\label{rem:twosided}
While we restricted ourselves here to sequents with contexts on the left, one can also consider one-sided sequents of the form $\vdash A, \Gamma$ (suitably adapting the definition of the categories $\SEQ_{\C S}(\vec D)$).
In this case the associated functor is of the form $\Phi(\Gamma):\CTX_{\C S} \longrightarrow \SEQ_{\C S} $ (notice that there is no opposite category in this case). Indeed, from a list of derivations $\vec\Pi:\vec E\longrightarrow \vec F$ one now obtains using $cut$ and $c$ a derivation of $(\vdash A, \vec E) \longrightarrow (\vdash A, \vec F)$.
By applying the same reasoning, one can also consider two-sided sequents of the form $\Gamma \vdash A, \Delta$, yielding two-variables functors 
$\Phi(\Gamma,\Delta):\CTX^{\mathsf{op}}_{S}\times \CTX_{S} \longrightarrow \SEQ_{\C S}/\equiv^{\mathrm{perm}}$, where the contravariant (i.e.~reversing) part corresponds to $\Gamma$ and the co-variant (i.e.~non-reversing) part corresponds to $\Delta$. 
\end{remark}

The rule ($\land$R) yields a way to pass from the schematic expressions $\Gamma \vdash A$ and $\Delta \vdash B$ to the schematic expression $\Gamma, \Delta \vdash A\land B$. 
Now, from the fact that we consider derivations modulo free permutations of rules, we will deduce that the rule ($\land$R) gives rise to a natural transformation 
$\C R: \Phi_{A}(\Gamma), \Phi_{B}(\Delta) \To \Phi_{A\land B}(\Gamma, \Delta)$.\footnote{Formally, we are here considering $\Phi_{A},\Phi_{B}$ and $\Phi_{A\land B}$ as functors onto the multicategory $\MSEQ_{\C S}/\equiv^{\mathrm{perm}}$, and using Remark \ref{rem:multinat}.}


First, we can observe that the rule ($\land$R) gives rise to multi-arrows $\C R_{\vec D, \vec E}:(\vec D\vdash A, \vec E \vdash B) \longrightarrow (\vec D,\vec E \vdash A\land B)$ in $\MSEQ_{\C S} $ (indexed by all possible lists of formulae $\vec D, \vec E$, i.e.~the possible values of the meta-variable $\Gamma$ and $\Delta$), where the derivation $\C R_{\vec D, \vec E}$ is made of an appropriate instance of ($\land$R):
\begin{center}
\adjustbox{scale=0.7}{\begin{minipage}{1.4\textwidth}\begin{align}
\C R_{\vec D,\vec E} \quad := \quad \AXC{}\dashedLine\UIC{$\vec D \vdash A$}
\AXC{}\dashedLine\UIC{$\vec E \vdash B$}
\RL{$\land$R}
\BIC{$\vec D,\vec E \vdash A\land B$}
\DP\end{align}\end{minipage}}\end{center}
Now, the requirement that this family be natural is expressed by an equational condition on the derivations $\C R_{\vec D,\vec E}$ which reads as follows:
given a list of derivations $\vec \Pi=\Pi_{1},\dots, \Pi_{n}: \vec E \longrightarrow \vec F$, the same proof is denoted by the derivation obtained by applying $\C R_{\vec D,\vec F}$ followed by the functorial derivation $\Phi_{A\land B}(\vec D, \vec\Pi)$, and by the derivation obtained by applying $\C R_{\vec D,\vec E}$ to the conclusions of the functorial derivation $\Phi_{B}(\vec \Pi)$, as illustrated below:

\begin{center}
\adjustbox{scale=0.7}{\begin{minipage}{1.4\textwidth}\begin{align}\label{natu1}
\AXC{}\dashedLine\UIC{$\vec D\vdash A$}
\AXC{}\dashedLine\UIC{$\vec F\vdash B$}
\RL{$\C R_{\vec D,\vec F}$}
\BIC{$\vec D,\vec F \vdash A\land B$}
\noLine
\UIC{\small$\Phi_{A\land B}(\vec D, \vec \Pi)$}
\noLine
\UIC{$\vec D,\vec E \vdash A\land B$}
\DP
\quad 
\equiv 
\quad
\AXC{}\dashedLine\UIC{$\vec D\vdash A$}
\AXC{}\dashedLine\UIC{$\vec F \vdash B$}
\noLine
\UIC{\small$\Phi_{B}(\vec \Pi)$}
\noLine
\UIC{$\vec E \vdash  B$}
\RL{$\C R_{\vec D,\vec E}$}
\BIC{$\vec D,\vec E \vdash A\land B$}
\DP\end{align}\end{minipage}}\end{center}In other words, applying an instance of the rule ($\land$R) with contexts $\Gamma=\vec D$ and $\Delta=\vec F$ followed by rules only acting on $\vec F$ (i.e.~for which the conclusion $A\land B$ is inactive) yields the same proof as applying an instance of the rule ($\land$R) with contexts $\Gamma=\vec D$ and $\Delta=\vec E$ \emph{after} this context is produced by rules for which the conclusions $A$ and $B$ are inactive.

If, as before, we let $\vec D=A$, $\vec E=C\supset B,C$ and $\vec F=B$, and we let $\vec \Pi$ be as in Eq.~\eqref{der2}, then the naturality condition \eqref{natu1}
precisely expresses the free permutation \eqref{der1}.

It is not difficult to be convinced that the instances of \eqref{natu1} express the equivalence of derivations up to free permutations of ($\land$R) with any sequence of other rules (those contained in $\Phi_{A\land B}(\vec D,\vec \Pi)$, which coincide with those in $\Phi_{B}(\vec \Pi)$). In other words, one can check that any instance of \eqref{natu1} can be deduced by applying a finite number of instances of $\equiv^{\mathrm{perm}}$ (and possibly of $\equiv^{\mathrm{cut}}$), leading to the following:
\begin{proposition}\label{prop1}
All equations \eqref{natu1} hold in ${\MSEQ_{\C S} }/ {\equiv^{\mathrm{perm}}}$. Thus, $\C R_{\vec D, \vec E}: \Phi_{A}(\Gamma),\Phi_{B}(\Gamma)\To \Phi_{A\land B}(\Gamma)$ is a natural transformation.
\end{proposition}

On the one hand, using Proposition \ref{prop1} we can say that the interpretation of sequent calculus rules as natural transformations is sound when derivations are considered up to free permutations of rules.
On the other hand, the permutation of ($\land$R) with an instance of some schematic rule 
\begin{center}
\adjustbox{scale=0.7}{\begin{minipage}{1.4\textwidth}\begin{align}
\AXC{$\Gamma, F_{1},\dots, F_{n}\vdash C$}
\RL{\small Rule}
\UIC{$\Gamma, E_{1},\dots, E_{n'}\vdash C$}
\DP
\end{align}\end{minipage}}\end{center}
for which $\Gamma$ and $C$ are inactive can be seen as an instance of \eqref{natu1}:
\begin{center}
\adjustbox{scale=0.7}{\begin{minipage}{1.4\textwidth}\begin{align}
\AXC{}\dashedLine\UIC{$\vec D\vdash A$}
\AXC{}\dashedLine\UIC{$\vec F\vdash B$}
\RL{$\C R_{\vec D,\vec F}$}
\BIC{$\vec D,\vec F \vdash A\land B$}
\RL{\small Rule ($\Gamma=\vec D$, $C=A\land B$)}
\UIC{$\vec D,\vec E \vdash A\land B$}
\DP
\quad 
\equiv 
\quad
\AXC{}\dashedLine\UIC{$\vec D\vdash A$}
\AXC{}\dashedLine\UIC{$\vec F \vdash B$}
\RL{\small Rule ($\Gamma$ empty, $C=B$)}
\UIC{$\vec E \vdash  B$}
\RL{$\C R_{\vec D,\vec E}$}
\BIC{$\vec D,\vec E \vdash A\land B$}
\DP\end{align}\end{minipage}}\end{center}

In this way we can thus argue that the full equivalence $\equiv^{\mathrm{perm}}$ is generated by stipulating naturality conditions in the style \eqref{natu1} for all rules of the calculus (see Remark below).
In other words, stipulating that sequent calculus rules yield natural transformations amounts at precisely the same as asserting that derivations only differing by free permutations of rules denote the same proof.

\section{Generalized Elimination Rules and Permutative Conversions}\label{sec:disjunction}

As recalled in the last section, for the disjunction-free fragment of propositional intuitionistic logic, natural deduction provides a canonical representation of proofs: distinct sequent calculus derivations which only differ by free permutations of rules are represented by the same natural deduction derivation.
What is the problem with disjunction, then? 

\subparagraph{Disjunction-Elimination and Permutative Conversions}

The failure of canonicity for natural deduction can be ascribed to the disjunction-elimination rule ($\lor$E) 
\begin{center}
\adjustbox{scale=0.7}{
\begin{minipage}{1.4\textwidth}
\begin{align}\label{velim}
\AXC{$A\lor B$}
\AXC{$\stackrel{p}{A}$}
\noLine
\UIC{$C$}
\AXC{$\stackrel{q}{B}$}
\noLine
\UIC{$C$}
\RL{$\lor$E $(p,q)$}
\TIC{$C$}
\DP\end{align}\end{minipage}}\end{center}
To see why, consider the two sequent calculus derivation schemas below, which are related by a free permutation of the rules ($\lor$L) and ($\land$R):\footnote{Notice that the permutation, read from left to right, duplicates the sub-derivation $\Sigma$).}\begin{center}
\adjustbox{scale=0.7}{
\begin{minipage}{1.4\textwidth}
\begin{align}\label{der4}
\AXC{$\Pi_{1}$}
\noLine
\UIC{$ A \vdash C$}
\AXC{$\Pi_{2}$}
\noLine
\UIC{$ B \vdash C$}
\RL{$\lor$L}
\BIC{$A\lor B \vdash C$}
\AXC{$\Sigma$}
\noLine
\UIC{$\vdash D$}
\RL{$\land$R}
\BIC{$A\lor B \vdash C\land D$}
\DP
\qquad \qquad
\AXC{$\Pi_{1}$}
\noLine
\UIC{$ A \vdash C$}
\AXC{$\Sigma$}
\noLine
\UIC{$\vdash D$}
\RL{$\land$R}
\BIC{$A \vdash C\land D$}
\AXC{$\Pi_{2}$}
\noLine
\UIC{$ B \vdash C$}
\AXC{$\Sigma$}
\noLine
\UIC{$\vdash D$}
\RL{$\land$R}
\BIC{$ B \vdash C\land D$}
\RL{$\lor$L}
\BIC{$A\lor B \vdash C\land D$}
\DP\end{align}\end{minipage}}\end{center}
When we transpose these schemas in natural deduction (with ($\lor$L) translating as ($\lor$E) and ($\land$R) translating as ($\land$I)), we still obtain two distinct derivations, since one can still choose between applying ($\land$I) before or after ($\lor$E):
\begin{center}
\adjustbox{scale=0.7}{
\begin{minipage}{1.4\textwidth}
\begin{align}\label{der5}
\AXC{$\stackrel{s}{A\lor B}$}
\AXC{$\stackrel{p}{A}$}
\noLine
\UIC{$\Pi_{1}^{*}$}
\noLine
\UIC{$  C$}
\AXC{$\stackrel{q}{B}$}
\noLine
\UIC{$\Pi_{2}^{*}$}
\noLine
\UIC{$ C$}
\RL{$\lor$E $(p,q)$}
\TIC{$C$}
\AXC{$\Sigma^{*}$}
\noLine
\UIC{$ D$}
\RL{$\land$I}
\BIC{$ C\land D$}
\DP
\qquad \qquad
\AXC{$\stackrel{s}{A\lor B}$}
\AXC{$\stackrel{p}{A}$}
\noLine
\UIC{$\Pi_{1}^{*}$}
\noLine
\UIC{$  C$}
\AXC{$\Sigma^{*}$}
\noLine
\UIC{$ D$}
\RL{$\land$I}
\BIC{$ C\land D$}
\AXC{$\stackrel{q}{B}$}
\noLine
\UIC{$\Pi_{2}^{*}$}
\noLine
\UIC{$ C$}
\AXC{$\Sigma^{*}$}
\noLine
\UIC{$ D$}
\RL{$\land$I}
\BIC{$ C\land D$}
\RL{$\lor$E $(p,q)$}
\TIC{$C\land D$}
\DP\end{align}\end{minipage}}\end{center}

Hence, in presence of the rule ($\lor$E), free permutations of rules have to be considered also for natural deduction derivations. 

A well-studied special case of this problem is that of permutations between ($\lor$E) and \emph{elimination} rules\footnote{More precisely, of elimination rules whose \emph{major} premiss is the conclusion of ($\lor$E).}: in this case the free permutation (oriented from left to right as below) is generally called a \emph{permutative conversion} \citep{Prawitz1971}: 
 \begin{center}
\adjustbox{scale=0.7}{
\begin{minipage}{1.4\textwidth}
\begin{align}\label{der6}
\AXC{$A\lor B$}
\AXC{$\stackrel{p}{A}$}
\noLine
\UIC{$\Pi_{1}$}
\noLine
\UIC{$  C$}
\AXC{$\stackrel{q}{B}$}
\noLine
\UIC{$\Pi_{2}$}
\noLine
\UIC{$ C$}
\RL{$\lor$E $(p,q)$}
\TIC{$C$}
\RL{\small elimination rule}
\UIC{$ C'$}
\DP
\qquad \leadsto \qquad
\AXC{$A\lor B$}
\AXC{$\stackrel{p}{A}$}
\noLine
\UIC{$\Pi_{1}$}
\noLine
\UIC{$  C$}
\RL{\small elimination rule}
\UIC{$ C'$}
\AXC{$\stackrel{q}{B}$}
\noLine
\UIC{$\Pi_{2}$}
\noLine
\UIC{$ C$}
\RL{\small elimination rule}
\UIC{$ C'$}
\RL{$\lor$E $(p,q)$}
\TIC{$C'$}
\DP\end{align}\end{minipage}}\end{center}Permutative conversions are necessary to extend to the case of disjunction the usual (desirable) fact that derivations in normal form enjoy the \emph{subformula property}. This is the property by which any formula occurring in the derivation is either a subformula of the conclusion or of some of the open assumptions. For instance, without appeal to permutative conversions, the derivation below left would be in normal form and would fail to satisfy the subformula property (since the formula $A\land C$ is a subformula of neither of the premises nor of the conclusion of the derivation). Instead, in presence of permutative conversions, this derivation is not normal and, by applying a permutative conversion followed by two $\beta$-conversions on the peaks ($\land$I)-($\land$E$_{1}$), one can reduce it to the derivation below right, which does satisfy the subformula property.
 \begin{center}
\adjustbox{scale=0.7}{
\begin{minipage}{1.4\textwidth}
\begin{align}
\AXC{$A\lor A$}
\AXC{$\stackrel{p}{A}$}
\AXC{$\stackrel{s}{C}$}
\RL{$\land$I}
\BIC{$A\land C$}
\AXC{$\stackrel{q}{A}$}
\AXC{$\stackrel{s}{C}$}
\RL{$\land$I}
\BIC{$A\land C$}
\RL{$\lor$E $(p,q)$}
\TIC{$A\land C$}
\RL{$\land$E$_{1}$}
\UIC{$A$}
\DP
\qquad
\qquad
\AXC{$A\lor A$}
\AXC{$\stackrel{p}{A}$}
\AXC{$\stackrel{q}{A}$}
\RL{$\lor$E $(p,q)$}
\TIC{$A$}
\DP\end{align}\end{minipage}}\end{center}

The usual property that normal forms exist and are unique is preserved
when permutative conversions are added to standard $\beta$ and $\eta$-conversions (see \citep{Prawitz1971}).  
However, this property is lost if, rather than just permutations of ($\lor$E) with elimination rules, one admits permutations of ($\lor$E) with \emph{any} any other rule (see \citep{Lindley2007}). 

What about the identity of proofs relation induced by these permutations (plus $\beta$- and $\eta$-conversions)? It is decidable, and if we restrict ourselves to standard permutative conversions, decidability is a simple consequence of existence and unicity of normal forms. Instead, if we consider all permutations, decidability becomes much more challenging to establish, and a complete proof was only published recently \citep{Scherer2017}, using the technique of \emph{focusing}.

The rule ($\lor$E) is one example from a larger family of natural deduction rules, called \emph{generalized elimination rules} \citep{Plato2001}, whose study has a long tradition in proof-theory.  
These rules are characterized by the presence of a meta-variable $C$ for an arbitrary possible conclusion occurring both as the conclusion of the rule and as the conclusion of some minor premises of the rule.
Other examples of generalized elimination rules, which can be seen as alternatives for usual elimination rules for $\land$ and $\supset$, are illustrated below,
 \begin{center}
\adjustbox{scale=0.7}{
\begin{minipage}{1.4\textwidth}
\begin{align}
\AXC{$A\land B$}
\AXC{$\stackrel{p}{A},\stackrel{q}{B}$}
\noLine
\UIC{$  C$}
\RL{$\land$E$_{\mathrm{gen}}$ $(p,q)$}
\BIC{$C$}
\DP
\qquad\qquad\qquad 
\AXC{$A\supset B$}
\AXC{$A$}
\AXC{$\stackrel{p}{B}$}
\noLine
\UIC{$  C$}
\RL{$\supset$E$_{\mathrm{gen}}$ $(p)$}
\TIC{$C$}
\DP\end{align}\end{minipage}}\end{center}
All generalized elimination rules give rise to problems similar to ($\lor$E), i.e.~free permutation of rules and the need of permutative conversions to preserve the subformula property. For the sake of readability, we here restrict our discussion to the case of ($\lor$E), but the analysis of permutations we develop in the following lines can be easily adapted to other generalized elimination rules.

\subparagraph{Sequents as Functors, Generalized Rules as Natural Transformations}


As we observed, the main feature of ($\lor$E) is that it makes use of a meta-variable $C$, which stands for an arbitrary conclusion. This meta-variable plays a role very similar to the context meta-variables $\Gamma,\Delta$ occurring in sequent calculus rules. We can then expect that the principle that two derivations differing by a free permutation involving ($\lor$E) denote the same proof can be 
expressed as a naturality condition relying on the schematic nature of $C$. We will show that this is actually the case.

In sequent calculus we can capture permutations like \eqref{der4} following the analysis from the previous section: focusing on the meta-variable $C$ in the rule ($\lor$L) 
 \begin{center}
\adjustbox{scale=0.7}{
\begin{minipage}{1.4\textwidth}
\begin{align}
\AXC{$\Gamma, A \vdash C$}
\AXC{$\Gamma, B \vdash C$}
\RL{$\lor$L}
\BIC{$\Gamma, A\lor B\vdash C$}
\DP
\end{align}
\end{minipage}}
\end{center}
the expressions $\Gamma, A\vdash C,\Gamma, B\vdash C$ and $\Gamma, A\lor B\vdash C$ yield (for any fixed choice $\Gamma=\vec D$) functors $\Phi_{A}(\vec D,\_), \Phi_{B}(\vec D,\_), \Phi_{A\lor B}(\vec D,\_): \CTX_{\C S}\longrightarrow \mathsf{(Multi)Sequents}_{\C S}/\equiv^{\mathrm{Perm}}$,
where $\Phi_{A}(\vec D, C)= \vec D, A\vdash C$, $\Phi_{B}(\vec D, C)=\vec D, B\vdash C$ and $\Phi_{A\lor B}(\vec D, C)=\vec D, A\lor B \vdash C$. Moreover, the rule ($\lor$L) itself yields a natural transformation
$\C S_{\vec D, \_}:\Phi_{A}(\vec D,\_), \Phi_{B}(\vec D,\_)\To \Phi_{A\lor B}(\vec D,\_)$ where
$\C S_{\vec D,C}$ is formed by an appropriate instance of ($\lor$L): \begin{center}
\adjustbox{scale=0.7}{
\begin{minipage}{1.4\textwidth}
\begin{align}
\C S_{\vec D,C} \quad := \quad
\AXC{}
\dashedLine
\UIC{$\vec D, A \vdash C$}
\AXC{}
\dashedLine
\UIC{$\vec D, B \vdash C$}
\RL{$\lor$L}
\BIC{$\vec D, A\lor B\vdash C$}
\DP
\end{align}
\end{minipage}}
\end{center}
The naturality condition for $\C S_{\vec D, C}$ reads as the permutation below, for any derivation $\Pi: E\vdash F$
\begin{center}
\adjustbox{scale=0.7}{\begin{minipage}{1.4\textwidth}\begin{align}\label{natu1bis}
\AXC{$\vec D, A\vdash E$}
\AXC{$\vec D, B\vdash E$}
\RL{$\C S_{\vec D,E}$}
\BIC{$\vec D,A \lor B \vdash E$}
\noLine
\UIC{\small$\Phi_{A\lor B}(\vec D, \Pi)$}
\noLine
\UIC{$\vec D,A\lor B \vdash F$}
\DP
\quad 
\equiv 
\quad
\AXC{$\vec D,A\vdash E$}
\noLine
\UIC{\small$\Phi_{A}(\vec D,\Pi)$}
\noLine
\UIC{$\vec D, A \vdash  E$}
\AXC{$\vec D, B \vdash E$}
\noLine
\UIC{\small$\Phi_{B}(\vec D,\Pi)$}
\noLine
\UIC{$\vec D, B \vdash  F$}
\RL{$\C S_{\vec D,F}$}
\BIC{$\vec D,A\lor B \vdash F$}
\DP\end{align}\end{minipage}}\end{center}
of which \eqref{der4} can be seen as an instance.

By reasoning as in the previous section, and considering that $\Phi_{A\lor B}(\vec D, \Pi)$ is composed of the same sequence of rules as $\Phi_{A}(\vec D, \Pi)$ and $\Phi_{B}(\vec D, \Pi)$, we can deduce that the permutation \eqref{natu1bis}, for any possible $\Pi$, is obtained by a sequence of free permutations of rules with ($\lor$L). This leads to the following:

\begin{proposition}\label{prop:natu1bis}
All equations \eqref{natu1bis} hold in ${\MSEQ_{\C S} }/ {\equiv^{\mathrm{perm}}}$. Thus, $\C S_{\vec D, \_}: \Phi_{A}(\vec D,\_),\Phi_{B}(\vec D,\_)\To \Phi_{A\lor B}(\vec D,\_)$ is a natural transformation.
\end{proposition}
%

On the other hand, the permutation of ($\lor$L) with an instance of some rule 
\begin{center}
\adjustbox{scale=0.7}{\begin{minipage}{1.4\textwidth}\begin{align}
\AXC{$\Gamma\vdash E$}
\RL{\small Rule}
\UIC{$\Gamma \vdash F$}
\DP
\end{align}\end{minipage}}\end{center}
for which $\Gamma$ is inactive can be seen as an instance of \eqref{natu1bis}:
\begin{center}
\adjustbox{scale=0.7}{\begin{minipage}{1.4\textwidth}\begin{align}\label{natuzz}
\AXC{}\dashedLine\UIC{$\vec D,A\vdash E$}
\AXC{}\dashedLine\UIC{$\vec D,B\vdash E$}
\RL{$\C S_{\vec D,E}$}
\BIC{$\vec D,A\lor B \vdash E$}
\RL{\small Rule ($\Gamma=\vec D,A\lor B$)}
\UIC{$\vec D,A\lor B \vdash F$}
\DP
\quad 
\equiv 
\quad
\AXC{}\dashedLine\UIC{$\vec D,A\vdash E$}
\RL{\small Rule ($\Gamma=\vec D,A$)}
\UIC{$\vec D,A\vdash F$}
\AXC{}\dashedLine\UIC{$\vec D,B \vdash E$}
\RL{\small Rule ($\Gamma=\vec D,B$)}
\UIC{$\vec D,B\vdash F$}
\RL{$\C S_{\vec D,F}$}
\BIC{$\vec D,A\lor B \vdash F$}
\DP\end{align}\end{minipage}}\end{center}
Hence, stipulating the naturality condition \eqref{natu1bis} to hold amounts at the same as asserting that derivations only differing by free permutations of rules with ($\lor$L) denote the same proof.

\subparagraph{Permutations in Natural Deduction, directly}
We will now discuss how permutations like \eqref{der5} or \eqref{der6} can be described by reasoning directly in natural deduction. For this we will exploit the representation of natural deduction derivations ``with holes''
as arrows of the category $\MSEQ_{\C S}$ described in Section \ref{sec:categories}.

The idea is that an instance of the rule ($\lor$E), with conclusion $D$, can be interpreted as a derivation $\Pi_{C}[\Sigma^{?},\Xi^{?}]: A\lor B \vdash C$ with two holes corresponding to unknown derivations of $C$ from assumptions $A$ and $B$, respectively: 
\begin{center}
\adjustbox{scale=0.7}{\begin{minipage}{1.4\textwidth}\begin{align}\label{velim2}
\Pi_{C}[\Sigma^{?},\Xi^{?}] \quad := \quad 
\AXC{$A\lor B$}
\AXC{$\stackrel{p}{A}$}
\noLine
\UIC{$\Sigma^{?}$}
\noLine
\UIC{$C$}
\AXC{$\stackrel{q}{B}$}
\noLine
\UIC{$\Xi^{?}$}
\noLine
\UIC{$C$}
\RL{$\lor$E $(p,q)$}
\TIC{$C$}
\DP
\end{align}\end{minipage}}\end{center}

The family of derivations $\Pi_{C}[\Sigma^{?},\Xi^{?}]$, indexed by all formulas $C$, yields then a natural transformation $\Pi_{C}[\Sigma^{?},\Xi^{?}]: \C K_{A\lor B}^{\vec D}(C), \Psi_{A}^{\vec D}(C),\Psi_{B}^{\vec D}(C) \longrightarrow \C I^{\vec D}(C)$, where the functors $\C K^{\vec D}_{A\lor B}(C), \Psi^{\vec D}_{A}(C),\Psi^{\vec D}_{B}(C), \C I^{\vec D}: \FORM_{\C S}(\vec D)\longrightarrow \mathsf{(Multi)Sequents}_{\C S}(\vec D)/\equiv^{\mathrm{Perm}}$ are defined as follows:
\begin{itemize}
\item $\C K^{\vec D}_{A\lor B}$ is the constant functor given by $\C K^{\vec D}_{A\lor B}(C)=\vec D\vdash A\lor B$ and $\C K^{\vec D}_{A\lor B}(\Lambda)=\AXC{}\dashedLine\UIC{${\vec D}\vdash A\lor B$}\DP$
(which in natural deduction form corresponds to the derivation only made of the assumption $A\lor B$);

\item $\Psi^{\vec D}_{A}$ is given by $\Psi^{\vec D}_{A}(C)=\vec D,A \vdash C$ and, for $\Lambda: \vec D,E\vdash F$, by the derivation $\Psi_{A}^{\vec D}(\Pi): ([\vec D],A\vdash E)\longrightarrow ([\vec D],A\vdash F)$ below (where we illustrate on the right its representation as a natural deduction derivation with a hole $\Sigma^{?}$)
\begin{center}
\adjustbox{scale=0.7}{\begin{minipage}{1.15\textwidth}\begin{align}\label{fex}
\Psi_{A}^{\vec D}(\Lambda) \quad = \quad 
\AXC{}
\dashedLine
\UIC{$\vec D,A\vdash E$}
\AXC{$\Lambda$}
\noLine
\UIC{$\vec D, E\vdash F$}
\RL{$cut$}
\BIC{$\vec D,A\vdash F$}
\DP
\qquad \approx \qquad
\left(
\AXC{$\vec D$}
\AXC{$\vec D,A$}
\noLine
\UIC{$\Sigma^{?}$}
\noLine
\UIC{$E$}
\noLine
\BIC{$\Lambda$}
\noLine
\UIC{$F$}
\DP\right)
\end{align}\end{minipage}}\end{center}
\item $\Psi^{\vec D}_{B}$ is given by $\Psi_{B}^{\vec D}(C)=\vec D,B \vdash C$ and, for $\Lambda: E\vdash F$, by the derivation $\Psi_{B}^{\vec D}(\Lambda): ([\vec D],B\vdash E)\longrightarrow ([\vec D],B\vdash F)$ constructed similarly to \eqref{fex};

\item $\C I^{\vec D}$ is the functor described in Section \ref{sec:categories} with $\C I^{\vec D}(C)=\vec D\vdash C$ and 
$\C I^{\vec D}(\Lambda): ([\vec D]\vdash E)\longrightarrow ([\vec D]\vdash F)$ being as below
\begin{center}
\adjustbox{scale=0.7}{\begin{minipage}{1.15\textwidth}\begin{align}\label{fexx}
\C I^{\vec D}(\Lambda) \quad = \quad 
\AXC{}
\dashedLine
\UIC{$\vec D\vdash E$}
\AXC{$\Lambda$}
\noLine
\UIC{$\vec D, E\vdash F$}
\RL{$cut$}
\BIC{$\vec D\vdash F$}
\DP
\qquad \approx \qquad
\left(
\AXC{$\vec D$}
\AXC{$\vec D$}
\noLine
\UIC{$\Sigma^{?}$}
\noLine
\UIC{$E$}
\noLine
\BIC{$\Lambda$}
\noLine
\UIC{$F$}
\DP\right)
\end{align}\end{minipage}}\end{center}
\end{itemize}

%
%
%
%
%
%
%

%
The naturality condition for $\Pi_{D}[\Sigma^{?},\Xi^{?}]$ amounts to the stipulation that for any two formulas $E,F$ and derivations 
$\Sigma: [\vec D], A \vdash E$, 
$\Xi: [\vec D], B \vdash E$ and 
$\Lambda : [\vec D], E \vdash F$, the same proof is denoted by the derivation obtained by composing $\Pi_{E}[\Sigma, \Xi]$ with $\C I^{\vec D}(\Lambda)$ and by the derivation obtained by replacing the holes of $\Pi_{F}$
by the functorial derivations $\Psi^{\vec D}_{A}(\Sigma)$ and $\Psi^{\vec D}_{B\vdash}(\Xi)$. In natural deduction this permutation is as illustrated below (where we omit the assumptions $\vec D$ for readability):
 \begin{center}
\adjustbox{scale=0.7}{
\begin{minipage}{1.4\textwidth}
\begin{align}\label{natu3}
\AXC{$A\lor B$}
\AXC{$\stackrel{p}{A}$}
\noLine
\UIC{$\Sigma$}
\noLine
\UIC{$  E$}
\AXC{$\stackrel{q}{B}$}
\noLine
\UIC{$\Xi$}
\noLine
\UIC{$ E$}
\RL{$\lor$E $(p,q)$}
\TIC{$E$}
\noLine
\UIC{$\Lambda$}
\noLine
\UIC{$ F$}
\DP
\qquad \equiv \qquad
\AXC{$A\lor B$}
\AXC{$\stackrel{p}{A}$}
\noLine
\UIC{$\Sigma$}
\noLine
\UIC{$  E$}
\noLine
\UIC{$\Lambda$}
\noLine
\UIC{$ F$}
\AXC{$\stackrel{q}{B}$}
\noLine
\UIC{$\Xi$}
\noLine
\UIC{$  E$}
\noLine
\UIC{$\Lambda$}
\noLine
\UIC{$ F$}
\RL{$\lor$E $(p,q)$}
\TIC{$F$}
\DP\end{align}\end{minipage}}\end{center}

The validity in $\MSEQ_{\C S}/\equiv^{\mathrm{Perm}}$ of the permutation above follows then from Proposition \ref{prop:natu1bis}.
On the other hand, the naturality condition \eqref{natu3} precisely expresses the equivalence of natural deduction derivations up to permutations of rules with ($\lor$E), and thus fully characterizes the principle by which two natural deduction derivations differing by a free permutations with ($\lor$E) denote the same proof.

%
%
%
\section{Propositional Quantification and Parametricity}\label{sec:parametricity}

When a logical system includes either first- or second-order quantifiers, other families of free permutations of rules have to be considered (see for instance the appendix of \citep{Hughes2018}), like for example those illustrated below (where we suppose $x$ not to occur free in either $A$ and $B$).
 \begin{center}
\adjustbox{scale=0.7}{
\begin{minipage}{1.4\textwidth}
\begin{align}\label{forall1}
\AXC{}
\RL{$ax$}
\UIC{$ A \vdash A$}
\RL{$\forall$R}
\UIC{$ A \vdash \forall x.A$}
\AXC{}
\RL{$ax$}
\UIC{$B\vdash B$}
\RL{$\supset$L}
\BIC{$B, B\supset A \vdash \forall x.A$}
\DP
\qquad\qquad
\AXC{}
\RL{$ax$}
\UIC{$ A \vdash A$}
\AXC{}
\RL{$ax$}
\UIC{$B\vdash B$}
\RL{$\supset$L}
\BIC{$B, B\supset A \vdash A$}
\RL{$\forall$R}
\UIC{$ B,B\supset A \vdash \forall x.A$}
\DP
\end{align}
\end{minipage}}

\bigskip

\adjustbox{scale=0.7}{
\begin{minipage}{1.4\textwidth}
\begin{align}\label{exists1}
\AXC{}
\RL{$ax$}
\UIC{$A\vdash A$}
\RL{$\exists$L}
\UIC{$ \exists x.A \vdash A$}
\AXC{}
\RL{$ax$}
\UIC{$B\vdash B$}
\RL{$\land$R}
\BIC{$B, \exists x.A \vdash A\land B$}
\DP
\qquad\qquad
\AXC{}
\RL{$ax$}
\UIC{$ A \vdash A$}
\AXC{}
\RL{$ax$}
\UIC{$B\vdash B$}
\RL{$\land$R}
\BIC{$B, A \vdash A \land B$}
\RL{$\exists$L}
\UIC{$ B,\exists x.A \vdash A\land B$}
\DP
\end{align}
\end{minipage}
}
\end{center}
For these examples one can develop an analysis based on naturality conditions similar to the one from Section \ref{sec:permutations}.\footnote{Actually, in defining free permutations of rules one must also take care of the side-conditions for the rules ($\forall$R) and ($\exists$L) assuring that the \emph{eigen-variable} does not occur in the context.} 
Interestingly, if we only add universal quantification to disjunction-free (propositional) intuitionistic logic, natural deduction still provides a canonical proof-system. For instance, to the two derivations in \eqref{forall1} there corresponds the unique derivation below.
 \begin{center}
\adjustbox{scale=0.7}{
\begin{minipage}{1.4\textwidth}
\begin{align}
\AXC{$\stackrel{p}{B\supset A}$}
\AXC{$\stackrel{q}{B}$}
\RL{$\supset$E}
\BIC{$A$}
\RL{$\forall$I}
\UIC{$\forall x.A$}
\DP\end{align}\end{minipage}}\end{center}
As soon as existential quantification is added, canonicity is lost. Indeed, the $\exists$-elimination rule is usually formulated as a generalized elimination rule, so the discussion from the previous section applies. For instance, to the two derivations in \eqref{exists1} there correspond two distinct natural deduction derivations illustrated below:
\begin{center}
\adjustbox{scale=0.7}{
\begin{minipage}{1.4\textwidth}
\begin{align}
\AXC{$\stackrel{p}{\exists x.A}$}
\AXC{$\stackrel{s}{A}$}
\RL{$\exists$E $(s)$}
\BIC{$A$}
\AXC{$\stackrel{q}{B}$}
\RL{$\land$I}
\BIC{$A\land B$}
\DP
\qquad \qquad
\AXC{$\stackrel{p}{\exists x.A}$}
\AXC{$\stackrel{s}{A}$}
\AXC{$\stackrel{q}{B}$}
\RL{$\land$I}
\BIC{$A\land B$}
\RL{$\exists$E $(s)$}
\BIC{$A\land B$}
\DP\end{align}\end{minipage}}\end{center}
 Similarly, canonicity in natural deduction is lost as soon as we consider classical logic (either linear or not) instead of intuitionistic logic. Graphical formalisms capturing free permutations involving quantifiers in classical (linear) logic have been studied in the literature (either based on proof nets \citep{LLSS} or combinatorial proofs \citep{Hughes2006,Strass2021}).

Depending on the particular form of quantification (first- or second-order) and  on the more or less {extensional} stance we have towards identity of proofs, further ways of permuting rules involving quantifiers can be considered.
These new permutations can be justified, as we'll see, by a strict interpretation of the schematic nature of first- and second-order derivations that in the computer science community goes under the name of \emph{parametricity} \citep{Reynolds1983, Bainbridge1990}. 

 While the permutations arising from parametricity are not \emph{strictu sensu} free permutations of rules (as they might involve more complex patterns), we will show that criteria of identity of proofs based on them can still be expressed as naturality conditions.
In this section we will focus on rule permutations for propositional (i.e.~second order) quantification, and in the next section we will consider rule permutations for first-order quantification.

\subparagraph{The Parametricity of Second-Order Proofs}
What is a proof of a universally quantified formula like $\forall p.p \supset p$, where $p$ ranges over the domain of all propositions? The formula says that for any possible choice of a proposition $A$, it is true that $A$ implies $A$. A natural answer would thus seem to be that a proof of $\forall p.p\supset p$ should be some method producing, for each choice of a proposition $A$, a proof of $A\supset A$. 
Under the functional interpretation of proofs, this means a method producing, for each proposition (or type) $A$, a function from $A$ to $A$ itself. 

Certainly any given proposition comes with its own ways of proving $A\supset A$ (or its own class of functions from $A$ to $A$). Yet, as is clear by inspecting the rules for quantifiers, proving $\forall p.p\supset p$ does not amount at listing different proofs of $A\supset A$, for any possible $A$; instead, it amounts at proving that $p\supset p$ holds where $p$ is a variable standing for any possible proposition; in this way, for any instantiation of $p$ as $A$, we can obtain a proof of $A\supset A$ in a uniform way. In functional terms, this means having what is called a \emph{parametric} function from a variable type $p$ to itself: any instantiation of $p$ with a type $A$ yields a function from $A$ to $A$, but all such functions will behave ``in the same way'' \citep{Reynolds1983, Bainbridge1990, Hermida2014}.

Interestingly, it has been argued that the interpretation of second-order proofs as parametric functions provides a possible way-out from traditional philosophical arguments agains the so-called ``impredicativity'' of second order logic, see \citep{Longo1997, BSL}.

Under this parametric view, there seems to be not many ways of constructing a proof of $p\supset p$: such a proof must produce a function from $p$ to $p$, where $p$ indicates a proposition or type we know nothing about. The only actual choice is provided by the identity function, corresponding to the natural deduction derivation below
\begin{center}
\adjustbox{scale=0.7}{
\begin{minipage}{1.4\textwidth}
\begin{align}\label{unique}
\AXC{$\stackrel{n}{p}$}
\RL{$\supset$I $(n)$}
\UIC{$p\supset p$}
\RL{$\forall$I}
\UIC{$\forall p.p\supset p$}
\DP
\end{align}\end{minipage}}\end{center}
Indeed, a purely semantic reasoning based on the idea of parametricity leads to conclude that the derivation above provides the \emph{unique} way of proving $\forall p.p\supset p$. We will now see how this line of reasoning can be used to conclude that different derivations must denote the same proof.

\subparagraph{New Permutations from Parametricity}

The view of second-order proofs as parametric functions can be used to justify new ways of permuting rules, and thus to defend a broader notion of identity of proof.
For instance, consider the two derivations below:
\begin{center}
\adjustbox{scale=0.7}{
\begin{minipage}{1.4\textwidth}
\begin{align}\label{param1}
\AXC{$\stackrel{r}{B\supset C}$}
\AXC{$\stackrel{n}{\forall p.p\supset p}$}
\RL{$\forall$E}
\UIC{$B\supset B$}
\AXC{$\stackrel{q}{B}$}
\RL{$\supset$E}
\BIC{$B$}
\RL{$\supset$E}
\BIC{$C$}
\DP
\qquad\qquad
\AXC{$\stackrel{n}{\forall p.p\supset p}$}
\RL{$\forall$E}
\UIC{$C\supset C$}
\AXC{$\stackrel{r}{B\supset C}$}
\AXC{$\stackrel{q}{B}$}
\RL{$\supset$E}
\BIC{$C$}
\RL{$\supset$E}
\BIC{$C$}
\DP\end{align}\end{minipage}}\end{center}Both derivations are constructed out of one occurrence of ($\forall$E) and two occurrences of ($\supset$E); however, these rules not only occur in different order, but the two occurrences of ($\forall$E) employ distinct instantiations of the variable $p$ (as $B$ on the left and as $C$ on the right). 
So why should we regard these derivations as denoting the same proof?

Suppose, as we did above, that a proof of $\forall p.p\supset p$, since parametric, can only be one which corresponds, functionally, to the identity function. 
Under this assumption, however we replace the undischarged assumption with label $n$ in any of the derivations in \eqref{param1}, this will produce (after a few normalization steps) the same proof. This fact is best appreciated when the derivations are decorated with proof terms (see also \citep{Isomorphism}, pp.~16-17):
\begin{center}
\adjustbox{scale=0.7}{
\begin{minipage}{1.4\textwidth}
\begin{align}\label{conve}
\AXC{$h:{B\supset C}$}
\AXC{$f:{\forall p.p\supset p}$}
\RL{$\forall$E}
\UIC{$fB:B\supset B$}
\AXC{$b:{B}$}
\RL{$\supset$E}
\BIC{$fBb: B$}
\RL{$\supset$E}
\BIC{$h(fBb):C$}
\DP
\qquad\qquad
\AXC{$f:{\forall p.p\supset p}$}
\RL{$\forall$E}
\UIC{$fC:C\supset C$}
\AXC{$h:{B\supset C}$}
\AXC{$b:{B}$}
\RL{$\supset$E}
\BIC{$hb:C$}
\RL{$\supset$E}
\BIC{$fC(hb):C$}
\DP
\end{align}\end{minipage}}\end{center}
Parametricity warrants that $fB$ and $fC$ must be the identity functions $1_{B}:B\supset B$ and $1_{C}:C\supset C$, and thus the two derivations encode (for any $f,h,b,B$ and $C$) the same proof of $C$:
\begin{center}
\adjustbox{scale=0.7}{
\begin{minipage}{1.4\textwidth}
\begin{align}
h(fBb) =  h(1_{B}b)= hb= 1_{C}(hb) = fC(hb)
\end{align}\end{minipage}}\end{center}

Parametricity can also be used to justify permutations of rules involving the existential quantifier. For example, how many proofs do there exist of the proposition $\exists p.p$? Intuitively, there are infinitely many, since for any proof $\Pi$ of some proposition $A$, we can obtain a different derivation as below.
\begin{center}
\adjustbox{scale=0.7}{
\begin{minipage}{1.4\textwidth}
\begin{align}\label{der7}
\AXC{$\Pi$}
\noLine
\UIC{$A$}
\RL{$\exists$I}
\UIC{$\exists p.p$}
\DP\end{align}\end{minipage}}\end{center}
Nevertheless, we will now argue, by reasoning in a somehow \emph{extensional} way, that two derivations as below should denote the same proof:
\begin{center}
\adjustbox{scale=0.7}{
\begin{minipage}{1.4\textwidth}
\begin{align}\label{exists2}
\AXC{${A\supset B}$}
\AXC{$\Pi$}
\noLine
\UIC{$A$}
\RL{$\supset$E}
\BIC{$B$}
\RL{$\exists$I}
\UIC{$\exists p.p$}
\DP
\qquad \equiv 
\qquad
\AXC{$\Pi$}
\noLine
\UIC{$A$}
\RL{$\exists$I}
\UIC{$\exists p.p$}
\DP\end{align}\end{minipage}}\end{center}
The extensional assumption we make is the following: if two proofs $\Pi,\Sigma$ of some formula $A$ are not the same, then there must be some way of separating them; more formally, there must be a (quantifier-free) formula $C$ with two distinct proofs $\Theta_{1},\Theta_{2}$, and a derivation $\Theta$ of $C$ with assumption $A$ such that $\Theta$ composed with $\Pi$ normalizes to $\Theta_{1}$, while
$\Theta$ composed with $\Sigma$ normalizes to $\Theta_{2}$:
\begin{center}
\adjustbox{scale=0.7}{
\begin{minipage}{1.4\textwidth}
\begin{align}
\AXC{$\Pi$}
\noLine
\UIC{$A$}
\noLine
\UIC{$\Theta$}
\noLine
\UIC{$C$}
\DP
\quad \leadsto\quad
 \AXC{$\Theta_{1}$}
\noLine
\UIC{$C$}
\DP
\quad\not\equiv \quad 
 \AXC{$\Theta_{2}$}
\noLine
\UIC{$C$}
\DP
\quad
\reflectbox{$\leadsto$}
\quad
\AXC{$\Sigma$}
\noLine
\UIC{$A$}
\noLine
\UIC{$\Theta$}
\noLine
\UIC{$C$}
\DP
\end{align}\end{minipage}}\end{center}
Under this assumption, the argument goes like this: suppose there exists a proof $\Sigma:(\exists p.p) \vdash C$ (notice that $p$ cannot occur in $C$) separating the two derivations in \eqref{exists2}  in the sense just explained.
%
%
$\Sigma$ can easily be converted into a proof of $\forall p. p\supset C$ (basic logic exercise), hence into a parametric proof $\Sigma'$ of $p\vdash C$, i.e.~a parametric function from the ``unknown'' type $p$ to a ``known'' type $C$, where the latter in no way depends on $p$.
One can then be convinced that this function can in no way use its input of type $p$, that is $\Sigma'$ cannot use the assumption $p$ (and thus $\Sigma$ cannot use the assumption $\exists p.p$). We deduce then that composing $\Sigma$ with some derivation $\Lambda$ of $\exists p.p$ has the effect of deleting $\Lambda$, and thus $\Sigma$ cannot be separating: 
%
%
\begin{center}
\adjustbox{scale=0.7}{
\begin{minipage}{1.4\textwidth}
\begin{align}
\AXC{$A\supset B$}
\AXC{$\Pi$}
\noLine
\UIC{$A$}
\RL{$\supset$E}
\BIC{$B$}
\RL{$\exists$I}
\UIC{$\exists p.p$}
\noLine
\UIC{$\Sigma$}
\noLine
\UIC{$C$}
\DP
\quad \equiv\quad
 \AXC{$\Sigma$}
\noLine
\UIC{$C$}
\DP
\quad
\equiv\quad
\AXC{$\Pi$}
\noLine
\UIC{$A$}
\RL{$\exists$I}
\UIC{$\exists p.p$}
\noLine
\UIC{$\Sigma$}
\noLine
\UIC{$C$}
\DP
\end{align}\end{minipage}}\end{center}
Put in simple words, the two derivations in \eqref{exists2} must denote the same proof since any way of possibly distinguishing them should arise from some parametric function, but no parametric function can do the job.

Using the permutation above, it is not difficult to show that any two derivations  of $\exists p.p$ denote the same proof: given two derivations as below
\begin{center}
\adjustbox{scale=0.7}{
\begin{minipage}{1.4\textwidth}
\begin{align}
\AXC{$\Pi_{1}$}
\noLine
\UIC{$A$}
\RL{$\exists$I}
\UIC{$\exists p.p$}
\DP
\qquad\qquad
\AXC{$\Pi_{2}$}
\noLine
\UIC{$B$}
\RL{$\exists$I}
\UIC{$\exists p.p$}
\DP\end{align}\end{minipage}}\end{center}we can reason as illustrated below:
\begin{center}
\adjustbox{scale=0.7}{
\begin{minipage}{1.4\textwidth}
\begin{align}
\AXC{$\Pi_{2}$}
\noLine
\UIC{$B$}
\RL{$\exists$I}
\UIC{$\exists p.p$}
\DP
\qquad
\stackrel{\tiny\text{($\beta$-equiv.)}}{\equiv}
\qquad
\AXC{$\Pi_{2}$}
\noLine
\UIC{$B$}
\RL{$\supset$I}
\UIC{$A\supset B$}
\AXC{$\Pi_{1}$}
\noLine
\UIC{$A$}
\RL{$\supset$E}
\BIC{$B$}
\RL{$\exists$I}
\UIC{$\exists p.p$}
\DP
\qquad\stackrel{\tiny\text{\eqref{exists2}}}{\equiv} 
\qquad
\AXC{$\Pi_{1}$}
\noLine
\UIC{$A$}
\RL{$\exists$I}
\UIC{$\exists p.p$}
\DP\end{align}\end{minipage}}\end{center}

\begin{remark}
The permutations arising from parametricity can be related to those discussed in Section \ref{sec:disjunction}. 
As is well-known, disjunction $A\lor B$ can be defined in second-order logic by the formula $\forall X.(A\supset X)\supset (B\supset X)\supset X$. This gives rise to a translation of second-order intuitionistic logic with disjunction into its disjunction-free fragment \citep{Prawitz1965} which not only preserves provability but also translates normalization steps (i.e.~$\beta$-reduction) into normalization steps.
However, the picture changes when permutative conversions are added as reduction steps: as discussed at length in \citep{StudiaLogica,SL2}, these conversions do not translate into $\beta$-reductions in second-order logic, but into a sequence of $\beta$-reductions, $\eta$-expansions as well as conversions similar to \eqref{conve}, which can be justified using parametricity or (as we discuss below) by stipulating a suitable naturality condition to hold. 

%
\end{remark}

\subparagraph{Formulae as Functors, Proofs as Natural Transformations}
We will now show that the criterion for identity of proofs expressed by
rule permutations like \eqref{param1} and \eqref{exists2} can be expressed by a naturality condition, this time applied to proofs and not just to rules. The idea that second-order proofs correspond to natural transformations dates back (at least) to \citep{Bainbridge1990, Girard1992}.
We here only provide a sketch of how such naturality condition can be defined, a more detailed discussion can be found in \citep{StudiaLogica, SL2}.

We start by formalizing the kind of extensional way reasoning that we used above, by defining a suitable quotient on the category of formulae.  
For any two formulae $A,B$, let $\equiv_{A,B}^{\mathrm{ext}}$ be the equivalence relation over derivations $\Pi,\Sigma: [\vec D],A \vdash B$ defined as follows: $\Pi \equiv_{A,B}^{\mathrm{ext}} \Sigma$ if there exists a quantifier-free formula $C$ and a derivation 
$\Theta: A\supset B \vdash C$ such that the two derivations obtained by composing $\Theta$ and $\Pi$, and $\Theta$ and $\Sigma$, respectively, have the same normal form.\footnote{This equivalence is a variant of usual \emph{contextual equivalence} from typed $\lambda$-calculus, see \citep{Barendregt2013}, p.~90.} We let the reader check that this is indeed an equivalence relation, and that $\equiv_{A,B}^{\mathrm{cut}}$ is included in $\equiv_{A,B}^{\mathrm{ext}}$.
We will consider the quotient category $\FORM_{\C S}(\vec D)/\equiv^{\mathrm{ext}}$ induced by this family of equivalence relations.
%

%
%
%
%
%
%

The fundamental observation now is that a formula $A[p]$ containing only \emph{positive} occurrences of a propositional variable $p$ yields a functor $F[p]: \FORM_{\C S}(\vec D)\longrightarrow \FORM_{\C S}(\vec D)/\equiv^{\mathrm{ext}}$ (see \citep{StudiaLogica, PistoneCSL}).
We limit ourselves here to a few examples:
\begin{itemize}

\item a formula $A$ not containing $p$ yields the constant functor $\Phi_{A}(E)=A$, which associates a derivation $\Pi:[\vec D],E\vdash F$ with the obvious derivation $\Phi_{A}[\Pi]:[\vec D], A \vdash A$;

\item the formula $A[p]=p$ yields the identity functor $\Phi_{A}(E)=E$, which associates a derivation with itself (more precisely, with its $\equiv^{\mathrm{ext}}$-equivalence class);

\item the formula $A[p]= B\supset p$, where $p$ does not occur in $B$, yields the functor $\Phi_{A}(E)=B\supset E$, which associates a derivation $\Pi: [\vec D],E\vdash F$ with the ($\equiv^{\mathrm{ext}}$-equivalence class of the) derivation $\Phi_{A}[\Pi]: [\vec D], B\supset E\vdash B\supset F$ below (we formulate it in natural deduction for uniformity):
\end{itemize} 
\begin{center}
\adjustbox{scale=0.7}{
\begin{minipage}{1.4\textwidth}
\begin{align}
\Phi_{A}[\Pi] \quad := \quad 
\AXC{$G_{B}[E]=B\supset E$}
\AXC{$\stackrel{p}{B}$}
\RL{$\supset$E}
\BIC{$E$}
\AXC{$\vec D$}
\noLine
\BIC{$\Pi$}
\noLine
\UIC{$F$}
\RL{$\supset$I $(p)$}
\UIC{$G_{B}[F]=B\supset F$}
\DP\end{align}\end{minipage}}\end{center}

We will now show that a derivation $\Pi$, whose conclusion $A[p]$ and assumptions $A_{1}[p],\dots, A_{k}[p]$ only contain positive occurrences of $p$, yields a natural transformation between the associated functors.
It is clear that for any choice of a formula $E$, by instantiating $p$ as $E$, we obtain a derivation $\Pi_{E}: A_{1}[E],\dots, A_{k}[E] \vdash A[E]$. Now, the naturality condition for the family of derivations $\Pi_{E}$ reads as follows: for any two formulae $E,F$ and derivation $\Sigma: [\vec D],E\vdash F$, the same proof is denoted by the derivation obtained by composing $\Pi_{E}$ with the functorial derivation $A[\Sigma]: [\vec D],A[E]\vdash A[F]$ and by the derivation obtained by replacing the assumptions of $\Pi_{F}$ with the functorial derivations $A_{i}[\Sigma]: [\vec D],A_{i}[E]\vdash A_{i}[F]$; in other words, this condition corresponds to the validity of the general permutation below (where we omit the assumptions $\vec D$ for readability):
 \begin{center}
\adjustbox{scale=0.7}{
\begin{minipage}{1.4\textwidth}
\begin{align}\label{natu3}
\AXC{$A_{1}[E]$}
\AXC{$\dots$}
\AXC{$A_{k}[E]$}
\noLine
\TIC{$\Pi_{E}$}
\noLine
\UIC{$A[E]$}
\noLine
\UIC{$\Phi_{A}[\Sigma]$}
\noLine
\UIC{$A[F]$}
\DP
\qquad
\equiv
\qquad
\AXC{$A_{1}[E]$}
\noLine
\UIC{$\Phi_{A_{1}}[\Sigma]$}
\noLine
\UIC{$A_{1}[D]$}
\AXC{$\dots$}
\AXC{$A_{k}[E]$}
\noLine
\UIC{$\Phi_{A_{k}}[\Sigma]$}
\noLine
\UIC{$A_{k}[F]$}
\noLine
\TIC{$\Pi_{F}$}
\noLine
\UIC{$A[F]$}
\DP\end{align}\end{minipage}}\end{center}

The following result holds in the case of second-order intuitionistic logic.
\begin{proposition}\label{prop:parametricity}
If $\C F$ indicates second-order propositional intuitionistic logic (i.e.~System F, see \citep{Girard1989}), then all instances of \eqref{natu3} hold in $\FORM_{\C F}(\vec D)/\equiv^{\mathrm{ext}}$. Thus, the family of derivations $\Pi_{\_}: \Phi_{A_{1}\land\dots\land A_{k}}(\_) \To \Phi_{A}(\_)$ is a natural transformation.
\end{proposition}
\begin{proof}{(Sketch)}
The claim follows from the following two facts (for details, see \citep{StudiaLogica, PistoneCSL}): first, the equivalence $\equiv^{\mathrm{ext}}$ is the \emph{maximum consistent equational theory} of System F, i.e.~the maximum equational theory for which there exists at least two non-equivalent derivations of some formula; second, the equations \eqref{natu3}, together with usual $\beta$- and $\eta$-equations, yield a consistent equational theory (since this theory has non-trivial models, see \citep{Bainbridge1990}). 
\end{proof}

The permutations \eqref{param1} and \eqref{exists2} are both instances of the naturality condition \eqref{natu3}.
Firstly, \eqref{param1} is an instance of \eqref{natu3}, with $k=2$, $A_{1}[p]=\forall q.q\supset q$, $A_{2}[p]=A_{3}[p]=p$, with $\Pi_{C}$ 
 and $\Sigma: B\vdash C$ being the derivation below:
 \begin{center}
\adjustbox{scale=0.7}{
\begin{minipage}{1.4\textwidth}
\begin{align}
\Pi_{C}\quad = \quad 
\AXC{$\forall p.p\supset p$}
\RL{$\forall$E}
\UIC{$C\supset C$}
\AXC{$C$}
\RL{$\supset$E}
\BIC{$C$}
\DP
\qquad\qquad
\Sigma \quad = \quad 
\AXC{$\stackrel{r}{B\supset C}$}
\AXC{$B$}
\RL{$\supset$E}
\BIC{$C$}
\DP\end{align}\end{minipage}}\end{center}
%
Secondly, \eqref{exists2} is an instance of \eqref{natu3} with $k=1$, $A_{1}[p]=p$, $A_{2}[p]=\exists q.q$, and with $\Pi_{C}$ and $\Sigma: A\vdash B$ being the derivations below
 \begin{center}
\adjustbox{scale=0.7}{
\begin{minipage}{1.4\textwidth}
\begin{align}
\Pi_{C}\quad = \quad 
\AXC{$C$}
\RL{$\exists$I}
\UIC{$\exists p.p$}
\DP
\qquad\qquad
\Sigma \quad = \quad 
\AXC{$\stackrel{r}{A\supset B}$}
\AXC{$A$}
\RL{$\supset$E}
\BIC{$B$}
\DP\end{align}\end{minipage}}\end{center}

\begin{remark}
Following \citep{Bainbridge1990, Girard1992} one can lift the restriction to formulas containing only positive occurrences of propositional variables obtaining that a formula $A$ yields a functor $\Phi_{A}[p,q]: \FORM\times \FORM^{\mathsf{op}}\longrightarrow \FORM$ of mixed variance. The naturality condition \eqref{natu3} is replaced then by a suitable \emph{dinaturality} condition \citep{MacLane}. 
\end{remark}

\begin{remark}
It is not true in general that naturality or dinaturality conditions like \eqref{natu3} characterize the quotient $\equiv^{\mathrm{ext}}$. Yet, as shown in \citep{PistoneCSL}, this is the case in some suitable fragments of second-order intuitionistic logic.
\end{remark}

The equational theory arising from naturality conditions like \eqref{natu3} (and more generally \emph{di}naturality conditions) can be hard to study in a purely syntactic way. In particular, in \citep{PistoneCSL} it is shown that the resulting equational theory is undecidable, but is decidable if one restricts to suitable fragments of second-order intuitionistic logic (see also \citep{StudiaLogica, SL2} for further proof-theoretical consequences of this fact).

%
%
%
%
%
%
%
%
%
%
%
%
%
%
%
%
%
%
%
%

%

\section{First-Order Quantification and Equality}\label{sec:unification}

The central intuition behind the rule permutations discussed in the previous section was that second-order derivations can be interpreted as parametric functions over propositions. Turning our attention to first-order logic, \emph{prima facie}, it would seem that first-order derivations might correspond as well to parametric functions over individuals, and that parametricity can be similarly invoked to justify
new conditions for identity of proofs.

In this section we provide a sketch of a way to formalize this intuition: we introduce a naturality condition for first-order derivations, and discuss a few permutations of rules that this condition permits to justify. 
Contrarily to our discussion of parametricity in second-order logic, which relies on a robust semantic tradition, our discussion of naturality for first-order logic is, to our knowledge, new. 
Nevertheless, we take inspiration from a recent graphical approach to the formalization of first-order derivations (see \citep{Hughes2018, Strass2021}), which seems capable of justifying rule permutations like those we discuss below.

In the previous sections we started from some already well-known class of rule permutations, and we showed that the principle by which two derivations only differing by one of such permutations denote the same proof can be expressed by a naturality condition. In this section we follow a reverse path: we first introduce a naturality condition capturing the idea of first-order proofs as parametric functions, and we then illustrate a few rule permutations that can be shown to be identity-preserving as a consequence of this condition.
%


%

\subparagraph{Natural Deduction with Equality}
We will work in a first-order language with equality, using a natural deduction framework with introduction and elimination rules for $=$ as follows (see e.g.~\citep{ptlc1}, p.~123, or \citep{deQueiroz2014}):

 \begin{center}
\adjustbox{scale=0.7}{
\begin{minipage}{1.4\textwidth}
\begin{align}
\AXC{}
\RL{$=$I}
\UIC{$t=t$}
\DP
\qquad\qquad
\AXC{$t=u$}
\AXC{$P(t)$}
\RL{$=$E}
\BIC{$P(u)$}
\DP\end{align}\end{minipage}}\end{center}
Observe that the usual commutativity and transitivity of equality $t=u\vdash u=t$ and $t=u,u=v\vdash t=v$ can be proved using the rules above. 
Moreover, we will consider derivations up to the equivalence generated by the $\beta$- and $\eta$-rules below

 \begin{center}
\adjustbox{scale=0.7}{
\begin{minipage}{1.4\textwidth}
\begin{align}
\AXC{}
\RL{$=$I}
\UIC{$t=t$}
\AXC{$P(t)$}
\RL{$=$E}
\BIC{$P(t)$}
\DP
\quad \equiv_{\beta}\quad
\AXC{$P(t)$}
\DP
\qquad\qquad
\AXC{}
\RL{$=$I}
\UIC{$u=u$}
\DP
\quad\equiv_{\eta}\quad
\AXC{$t=u$}
\AXC{}
\RL{$=$I}
\UIC{$t=t$}
\RL{$=$E}
\BIC{$u=u$}
\DP\end{align}\end{minipage}}\end{center}
Weaker variants of the $\eta$-rule are considered in 
\citep{Helman1987}. A more refined approach to equality in natural deduction is presented in \citep{deQueiroz2014}.

By a \emph{substitution} we indicate a function $\theta$ from a finite set first-order variables, called the \emph{support} of $\theta$ and noted $\mathrm{supp}(\theta)$, to first-order terms. For any substitution $\theta$, we let $\Gamma_{\theta}$ be the finite set of equations $\{x_{1}=\theta(x_{1}),\dots, x_{k}=\theta(x_{k})\}$, where $\mathrm{supp}(\theta)=\{x_{1},\dots, x_{k}\}$. 
Observe that $\Gamma_{\theta}$ completely characterizes $\theta$.
Finally, for any first-order term $t$ or formula $A$ and substitution $\theta$, we let $t\theta$ (resp.~$A\theta$) be the result of replacing in $t$ (resp.~in $A$) any variable $x\in \mathrm{supp}(\theta)$ with $\theta(x)$.

\subparagraph{From Parametricity to First-Order Naturality}

Let us first make the analogy with the second-order case more explicit. In the previous section we explored the view that a proof of a universally quantified formula $\forall p.A$ corresponds to a parametric function producing, for any proposition $C$, a proof of $A[C]$, so that for different arguments $C$ and $D$, the proofs of $A[C]$ and $A[D]$ are constructed ``in the same way''.
When moving to a first-order setting, it makes perfect sense to say that also a proof of a universally quantified formula $\forall x.A$ should be thought as a parametric function yielding, for any first-order term $t$, a proof of $A[t]$ so that for different arguments $t$ and $u$, the resulting proofs of $A[t]$ and $A[u]$ are constructed ``in the same way''.
After all, standard rules for first-order and second-order universal quantifier follow the same pattern.

We also observed that the parametric view leads to identity all derivations of $\exists p.p$, using the fact that a proof of $(\exists p.p)\supset C$ (i.e.~of $\forall p.p\supset C$) is a parametric function. 
We can imagine a similar argument for first-order logic: take two distinct derivations of $\exists x. x=x$, e.g.~coming from different instances of the rule ($=$I):
 \begin{center}
\adjustbox{scale=0.7}{
\begin{minipage}{1.4\textwidth}
\begin{align}\label{fol1}
\AXC{}
\RL{$=$I}
\UIC{$t=t$}
\RL{$\exists$I}
\UIC{$\exists x.x=x$}
\DP
\qquad
\qquad
\AXC{}
\RL{$=$I}
\UIC{$u=u$}
\RL{$\exists$I}
\UIC{$\exists x.x=x$}
\DP\end{align}\end{minipage}}\end{center}
If we reason in an extensional way, in order to distinguish the two derivations above, we should be capable of constructing a derivation of $(\exists x.x=x)\supset C$ (i.e.~of $\forall x.(x=x\supset C)$), where $C$ is quantifier-free and does not contain $x$. Yet, since $x$ cannot occur in $C$, such a derivation (that we suppose to be in normal form) cannot actually employ the assumption $x=x$, and thus, by an argument similar to that of the previous section, this proof cannot distinguish between the two derivations above. We are thus led to the claim that there should be exactly one proof of $\exists x.x=x$, up to equivalence. How can we make these intuitive ideas precise?

\subparagraph{Formulas as Functors, Proofs as Natural Transformations}

The fundamental difference between the first-order and second-order case is that in the latter one is considering parametric propositions $A[p]$ acting over propositions themselves, that we conveniently described as functors $\FORM_{\C S}\longrightarrow \FORM_{\C S}/\equiv^{\mathrm{ext}}$; in the first-order case formulae are not dependent on propositions but on individual terms. We thus need to replace the category $\FORM_{\C S}$ in input with a suitable category for terms.

 We define the \emph{category of terms of $\C S$}, indicated as $\TERM_{\C S}$, as follows: its objects are first-order terms in the language of $\C S$, noted $t,u,v,\dots$; an arrow between two terms $t,u$ (noted $\theta:t\mapsto u$) is given by a substitution $\theta$ such that $u=t\theta$; diagrams are just plain equations between substitutions; the identity morphism $1_{t}$ is the identity substitution with support the set of free variables of $t$; the composition of substitutions $\theta$ and $\lambda$ is given by the substitution $\lambda *\theta$ such that for any term $t$, $t(\lambda*\theta)=(t\lambda)\theta$ (the reader can verify that $\lambda*\theta$ still has a finite support).

Moreover, we consider a slight variation $\FORM_{\C S}^{=}$ of the category $\FORM_{\C S}$: objects are always formulae of $\C S$, but an arrow of $\FORM_{\C S}^{=}$ between $A$ and $B$ is now a pair $(\B E,\Pi)$ made of a finite set of equations $\B E=\{t_{1}=u_{1},\dots, t_{n}=u_{n}\}$ together with a derivation of $\B E,A \vdash B$ (more precisely, a $\equiv^{\mathrm{cut}}$-equivalence class of derivations). The identity $1_{A}$ is the pair made of the empty set and the obvious derivation of $A\vdash A$; the composition of $(\B E, \Pi)$ and $(\B F, \Sigma)$ is made of $\B E\cup \B F$ together with the derivation obtained by applying a cut-rule to $\Pi$ and $\Sigma$.

The quotient $\equiv^{\mathrm{ext}}$ is extended from $\FORM_{\C S}$ to $\FORM_{\C S}^{=}$ by letting $(\B E, \Pi)\equiv^{\mathrm{ext}}_{A,B}(\B F,\Sigma)$ when $\B E=\B F$ and $\Pi\equiv^{\mathrm{ext}}_{A,B}\Sigma$ holds (seeing $\Pi$ and $\Sigma$ as arrows in $\FORM_{\C S}(\B E)$).

Let $P[x]$ be a first-order formula with possible occurrences of the variable $x$. We can associate with $P[x]$ a functor $ \TERM_{\C S} \longrightarrow \FORM_{\C S}^{=}/\equiv^{\mathrm{ext}}$ defined as follows: for any term $t$, we let $P[t]$ be the formula $P[x:=t]$; for any substitution $\theta: t\mapsto u$, we let $P[\theta]$ be the pair composed by $\Gamma_{\theta}$ together with the derivation
$\Pi(P,\theta): \Gamma_{\theta}, P[x:=t]\vdash P[x:=u]$ below
 \begin{center}
\adjustbox{scale=0.7}{
\begin{minipage}{1.4\textwidth}
\begin{align}
\Pi(P,\theta)\quad := \quad 
\AXC{$\Gamma_{\theta}=\{x_{1}=\theta(x_{1}),\dots, x_{n}=\theta(x_{n})\}$}
\AXC{$P[x:=t]$}
\doubleLine
\RL{$=$E}
\BIC{$P[x:=u]$}
\DP\end{align}\end{minipage}}\end{center}using the fact that $u=t[x_{1}:=\theta(x_{1}),\dots, x_{n}:=\theta(x_{n})]$.

We now show that a derivation $\Pi$ of premisses $P_{1}[x],\dots, P_{n}[x]$ and conclusion $P[x]$, yields a natural transformation between the associated functors. 
It is clear that for any choice of a term $t$, by instantiating $x$ as $t$ we obtain a derivation $\Pi_{t}: P_{1}[x:=t],\dots, P_{n}[x:=t]\vdash P[x:=t]$.
The naturality condition for the family $\Pi_{t}$ reads then as follows: for any two terms $t,u$ and substitution $\theta: t\mapsto u$, the same proof is denoted by the derivation obtained by composing $\Pi_{t}$ with the functorial derivation $\Pi(P,\theta): \Gamma_{\theta},P[x:=t]\vdash P[x:=u]$, and by the derivation obtained by replacing the assumptions of $\Pi_{u}$ with the functorial derivations $\Pi(P_{i},\theta):\Gamma_{\theta},P_{i}[x:=t] \vdash P_{i}[x:=t]$, as illustrated below (where we omit the open assumptions $\Gamma_{\theta}$ for readability):
 \begin{center}
\adjustbox{scale=0.7}{
\begin{minipage}{1.4\textwidth}
\begin{align}\label{natu4}
\AXC{$P_{1}[x:=t]$}
\AXC{$\dots$}
\AXC{$P_{n}[x:=t]$}
\noLine
\TIC{$\Pi_{t}$}
\noLine
\UIC{$P[x:=t]$}
\noLine
\UIC{$\Pi(P,\theta)$}
\noLine
\UIC{$P[x:=u]$}
\DP
\qquad
\equiv
\qquad
\AXC{$P_{1}[x:=t]$}
\noLine
\UIC{$\Pi(P_{1},\theta)$}
\noLine
\UIC{$P_{1}[x:=u]$}
\AXC{$\dots$}
\AXC{$P_{n}[x:=t]$}
\noLine
\UIC{$\Pi(P_{n},\theta)$}
\noLine
\UIC{$P_{n}[x:=u]$}
\noLine
\TIC{$\Pi_{u}$}
\noLine
\UIC{$P[x:=u]$}
\DP\end{align}\end{minipage}}\end{center}

For first-order intuitionistic logic we can prove the following fact:

\begin{proposition}\label{prop:parametricity}
If $\C G$ indicates first-order intuitionistic logic, then all instances of \eqref{natu4} hold in $\FORM_{\C G}^{=}(\vec D)/\equiv^{\mathrm{ext}}$. Thus, the family of derivations $\Pi_{\_ }: P_{1}[\_]\land \dots \land P_{n}[\_]\To P[\_]$ is a natural transformation.
\end{proposition}
\begin{proof}
The claim is proved similarly to Proposition \ref{prop:parametricity}. First, the equivalence $\equiv^{\mathrm{ext}}$ can be shown to be the {maximum consistent equational theory} of $\FORM_{\C G}^{=}$ by an argument similar to the case of System F. It suffices then to show that the equations \eqref{natu4}, together with usual $\beta$- and $\eta$-equations, yield a consistent equational theory.


%

Let $\equiv^{\mathrm{1nat}}$ indicate the equivalence relation on derivations generated from usual $\beta$- and $\eta$-equivalences together with all instances of \eqref{natu4}. We will show that $\equiv^{\mathrm{1nat}}$ is consistent. 
There exists a well-known ``forgetful'' translation from intuitionistic first-order logic into intuitionistic propositional logic, defined over formulas as follows:\footnote{We restrict ourselves for simplicity to the $\supset,\forall$-fragment, but the translation works well for all other propositional connectives.}
 \begin{center}
\adjustbox{scale=0.7}{
\begin{minipage}{1.4\textwidth}
\begin{align}
(t=u)^{*}= \mathbf{True}
\qquad
P(t)^{*} = P \qquad (A\supset B)^{*}= A^{*}\supset B^{*}  \qquad (\forall x.A)^{*}=A^{*} \qquad (\exists x.A)^{*}
\end{align}\end{minipage}}\end{center}
and where, for any first-order derivation $\Pi:\vec A \vdash A$, the propositional derivation $\Pi^{*}: \vec A^{*}\vdash A^{*}$ is obtained from $\Pi$ just by deleting all instances of ($=$E) as well as all instances of rules for quantifiers. 
Observe that, since any propositional derivation is also a first-order derivation, the translation is surjective: any propositional derivation is the translation of some first-order derivation.

Under this translation, one can easily check that two first-order derivations $\Pi,\Sigma$ which are equivalent under $\beta$- or $\eta$-equations are translated into two propositional derivations $\Pi^{*},\Sigma^{*}$ which are still equivalent under $\beta$- or $\eta$-equations. One can easily check then that all instances of the two derivations in \eqref{natu4} translate onto equivalent propositional derivations: on the one hand $(\Pi_{t})^{*}$ is the same derivation as $(\Pi_{u})^{*}$; on the other hand the functorial derivations $\Pi(P,\theta)$ and $\Pi(P_{i},\theta)$ all translate to the trivial derivation only consisting of one assumption. 

We conclude then that if $\Pi\equiv^{\mathrm{1nat}}\Sigma$ holds for some first-order derivations $\Pi,\Sigma$, then $\Pi^{*},\Sigma^{*}$ are propositional derivations equivalent under $\beta$- and $\eta$-equivalence.
Hence, if $\equiv^{\mathrm{1nat}}$ were inconsistent, since the translation $^{*}$ is surjective, then also the $\beta\eta$-equivalence of propositional intuitionistic logic would be inconsistent, something which is well-known not to be the case.

\end{proof}

\subparagraph{New Rule Permutations}

Let us see how the naturality condition \eqref{natu4} can be used to justify new criteria for identity of proofs in first-order logic. 
Firstly, we will justify the claim that the derivations \eqref{fol1} denote the same proof. We will then deduce from this that $\exists x.x=x$ admits precisely one proof.

For any term $t$, let $\theta_{x\mapsto t}$ be the substitution determined by $\Gamma_{\theta_{x\mapsto t}}=\{x=t\}$. 
Let $n=1$, $P_{1}[x]$ be $x=x$, $P[x]=\exists x.x=x$, and $\Pi_{x}$ be the derivation 
 \begin{center}
\adjustbox{scale=0.7}{
\begin{minipage}{1.4\textwidth}
\begin{align}
\Pi_{x}\quad := \quad 
\AXC{}\RL{$=$I}\UIC{$x=x$}
\RL{$\exists$I}\UIC{$\exists x.x=x$}\DP
\end{align}\end{minipage}}
\end{center}
We can then compute
 \begin{center}
\adjustbox{scale=0.7}{
\begin{minipage}{1.4\textwidth}
\begin{align}
\Pi_{x}\quad := \quad 
\AXC{}\RL{$=$I}\UIC{$x=x$}
\RL{$\exists$I}\UIC{$\exists x.x=x$}\DP
\quad \equiv\quad
\AXC{}\RL{$=$I}\UIC{$x=x$}
\RL{$\exists$I}\UIC{$\exists x.x=x$}
\noLine
\UIC{$\Pi(P,\theta_{x\mapsto t})$}
\noLine
\UIC{$\exists x.x=x$}
\DP
\quad \stackrel{\tiny\eqref{natu4}}{\equiv}\quad
\AXC{}\RL{$=$I}\UIC{$x=x$}
\noLine
\UIC{$\Pi(P_{1},\theta_{x\mapsto t})$}
\noLine
\UIC{$t=t$}
\RL{$\exists$I}\UIC{$\exists x.x=x$}\DP
\quad \stackrel{\tiny(\text{$\eta$-equiv.})}{\equiv}\quad
\AXC{}\RL{$=$I}\UIC{$t=t$}
\RL{$\exists$I}\UIC{$\exists x.x=x$}\DP
\quad := \quad 
\Pi_{t}
\end{align}\end{minipage}}\end{center}
using the fact that $\Pi(P,\theta_{x\mapsto t})$ is the identity derivation and that $\Pi(P_{1},\theta_{x\mapsto t})$ is as illustrated below
 \begin{center}
\adjustbox{scale=0.7}{
\begin{minipage}{1.4\textwidth}
\begin{align}\Pi(P_{1},\theta_{x\mapsto t})\quad = \quad \AXC{$x=t$}\AXC{}\RL{$=$I}
\UIC{$x=x$}
\RL{$=$E}
\BIC{$t=t$}
\DP\end{align}\end{minipage}}\end{center}In a similar way one can prove that $\Pi_{x}\equiv \Pi_{u}$, and thus conclude that $\Pi_{t}\equiv\Pi_{u}$.

By arguing in a similar way one can also justify a permutation like the one below, involving the universal quantifier:
 \begin{center}
\adjustbox{scale=0.7}{
\begin{minipage}{1.4\textwidth}
\begin{align}
\AXC{$t=u$}
\AXC{$\forall x.P(x)$}
\RL{$\forall$E}
\UIC{$P(t)$}
\RL{$=$E}
\BIC{$P(u)$}
\DP
\qquad
\equiv
\qquad
\AXC{$\forall x.P(x)$}
\RL{$\forall$E}
\UIC{$P(u)$}
\DP
\qquad\qquad (u=t\theta)
\end{align}\end{minipage}}\end{center}

More generally, \eqref{natu3} can be seen as a general permutative principle which can be used to permute occurrences of the rule ($=$E) with other rules. It could be interesting to study whether one can devise conversion rules for ($=$E) capturing \eqref{natu3} in terms of rewriting. 

\begin{remark}
We do not know whether the equivalence arising from \eqref{natu3} fully characterizes the quotient $\equiv^{\mathrm{ext}}$ in first-order logic. In any case, it might be interesting to develop a denotational semantics, in the style of the dinatural semantics of \citep{Bainbridge1990}, interpreting first-order proofs as natural transformations, and to check whether the naturality condition \eqref{natu3} fully characterizes, if not the full $\equiv^{\mathrm{ext}}$, at least the equivalence arising from this semantics. 

\end{remark}

\section{Conclusions}\label{sec:conclusion}

In this paper we explored the formalization of principles of identity of proofs by means of naturality conditions, suggesting that these conditions provide a finer mathematical understanding of the schematic nature of formal derivations.

The goal of this paper was mostly conceptual, that is, of showing that the invariance under
 different classes of rule permutations can be expressed by stipulating the validity of suitable naturality conditions. 
On the other hand, devising concrete and syntactically well-behaved equational theories capturing rule permutations can be difficult. 
For instance, when considering the equational theory arising from some permutative rule, one would like to be able to prove that, by orienting such rules, any derivation can be associated with a (unique) normal form (i.e.~a derivation to which no other oriented permutation can be applied). This has the desirable consequence of making the equational theory decidable (to know if two derivations are equivalent, check their normal forms) and also of providing a more insightful description of equivalence classes of derivations (which are in bijection with normal forms). 

However, at the current state of the arts, most of the permutative rules treated in this paper have so far resisted a purely syntactic analysis in terms of normal forms, and the resulting equational theories have in some cases been shown not to be decidable. For full propositional intuitionistic logic decidability was only shown recently \citep{Scherer2017}, while for second-order logic the identity of proof relations arising from either $\equiv^{\mathrm{ctx}}$ or from the (di)naturality conditions discussed in Section \ref{sec:parametricity} are both undecidable (see \citep{PistoneCSL}). In both cases, normal forms for the reduction obtained by orienting rule permutations need neither exist nor be unique (see \citep{Lindley2007} and \citep{StudiaLogica}).

To our knowledge, the equational theory for first-order logic sketched in Section \ref{sec:unification} has not been studied before in the literature. It certainly needs further technical investigation. We do not know if this theory is decidable, neither if it can profitably be studied using rewriting techniques (for instance, by defining conversion rules permuting instances of ($=$E) upwards).   
 Recent work on proof nets \citep{Hughes2018} and combinatorial proofs \citep{Strass2021} seems to lead to a similar way of identifying derivations of existentially quantified formulas, and we conjecture that the equational theory arising from these approaches might be related to the naturality condition \eqref{natu4}.

\bibliographystyle{apa}
\bibliography{NaturalityOfProofs.bib}

\end{document}